\documentclass{article}

\usepackage{amsthm,amssymb,amsfonts,mathrsfs}
\usepackage{amsmath}
\usepackage{color}
\catcode`@=11 \@addtoreset{equation}{section}

\catcode`@=12

\newtheorem{Theorem}{Theorem}[section]
\newtheorem{Proposition}{Proposition}[section]
\newtheorem{Lemma}{Lemma}[section]

\theoremstyle{definition}

\newtheorem{Definition}{Definition}[section]

\newtheorem{Assumptions}{Hypothesis}[section]

\def\R{{\mathbb{R}}}

\def\cU{{\mathcal{U}}}

\def\t\cU{{\widetilde{{\mathcal{U}}}}}

\newcommand\norm[1]{\left\lVert#1\right\rVert}

\title {Stability for some classes of degenerate nonlinear hyperbolic equations with time delay}
\author{{{\sc Alessandro Camasta}}\\
	Department of Mathematics\\ University of Bari Aldo Moro\\
	Via
	E. Orabona 4\\ 70125 Bari - Italy\\ e-mail: alessandro.camasta@uniba.it\\
	{\sc Genni Fragnelli}\\
	Department of Ecology and Biology\\ Tuscia University\\
	Largo dell'Universit\`a, 01100 Viterbo - Italy\\ e-mail: genni.fragnelli@unitus.it\\
	{\sc Cristina Pignotti}\\
Dipartimento di Ingegneria e Scienze dell'Informazione e Matematica\\ Universit\`a di L'Aquila\\
Via Vetoio, Loc. Coppito,  67100 L'Aquila - Italy\\ e-mail: cristina.pignotti@univaq.it
}

\begin{document}
	
	\maketitle
	\begin{abstract}
		We consider several classes of degenerate hyperbolic equations involving delay terms and
suitable nonlinearities. The idea is to rewrite the problems in an abstract way and, using semigroup theory and energy method, we study well posedness and stability. Moreover, some illustrative examples are given.
	\end{abstract}
	\noindent 
	Keywords: fourth order
degenerate operator, second order degenerate operator, operator in divergence or in non divergence form, exponential stability, nonlinear equation, time delay.
\vspace{0.2cm}

\noindent 2000AMS Subject Classification: 35L80, 93D23, 93D15, 93B05, 93B07

	\section{Introduction}
In this paper, we study well-posedness and  stability for the following degenerate problems with time delay and  nonlinear source:	
\begin{equation}\label{P0}
	\begin{cases}
		y_{tt}(t,x)-A_iy(t,x)+k(t)BB^*y_t(t-\tau,x)=f(y(t,x)), &(t,x)\in Q\\
		y(0,x)=y^0(x),\,\,y_t(0,x)=y^1(x),&x\in(0,1),\\
		\mathcal B_iy(t,0)=0, & t >0,\\
		\mathcal C_iy(t,1)=0, & t>0,\\
		B^*y_t(s,x)=g(s), &s\in [-\tau,0],
	\end{cases}
\end{equation}
$i =1,2,3,4$.
Here $Q:= (0, +\infty) \times (0,1),$ $\mathcal B_iy(t,0)=0$ and $\mathcal C_iy(t,1)=0$ are suitable boundary conditions related to the operators $A_i$, $i =1,2,3,4$.
In particular,
\[
A_iy:= \begin{cases}
	-a(x)y_{xxxx}(t,x), &i=1,\\
	-(ay_{xx})_{xx}(t,x), &i=2,\\
	a(x)y_{xx}(t,x), &i=3,\\
	(ay_{x})_x(t,x), &i=4,
\end{cases}
\]
\[
\mathcal B_iy(t,0):=\begin{cases} y(t,0)=y_x(t,0)=0, &i=1,\\
	y(t,0)= \begin{cases}
		y_x(t,0)=0, &\text{ if } a \text{ is (WD)},\\
		(ay_{xx})(t,0)=0, &\text{ if } a \text{ is (SD)},\\ 
	\end{cases} &i=2,\\
	y(t,0)=0, & i=3,\\
	\begin{cases}
		y(t,0)=0, &\text{ if } a \text{ is (WD)},\\
		\lim_{x\to 0}(ay_{x})(t,x)=0, &\text{ if } a \text{ is (SD)},\\ 
	\end{cases} &i=4
\end{cases}
\]
and
\[
\mathcal C_iy(t,1):=  \begin{cases} \begin{cases} \beta y(t,1)-y_{xxx}(t,1)+y_t(t,1)=0, \\
		\gamma y_x(t,1)+y_{xx}(t,1)+y_{tx}(t,1)=0,
	\end{cases} &i=1,\\
	\begin{cases}
		\beta y(t,1)-(ay_{xx})_x(t,1)+y_t(t,1)=0,\\
		\gamma y_x(t,1)+(ay_{xx})(t,1)+y_{tx}(t,1)=0, \\
	\end{cases} &i=2,\\
	\beta y(t,1)+y_{x}(t,1)+y_t(t,1)=0, & i=3,4.\\
\end{cases}
\]
The bounded linear operator $B$ that appears in \eqref{P0} is defined on a real Hilbert space with adjoint $B^*$ and $f$ is a nonlinear function satisfying suitable hypotheses that will be specified in the next sections. 
%As an example for $B$ and $f$ we can consider, respectively, $BB^*y=\chi_\omega y$, with $\omega \subseteq [0,1]$, and $f(y)=|y|^q y$, $q>0$, or $f(y)=(\int_0^1|y|^2dx)^{\frac{p}{2}}y$, $p\ge 1$. 
Moreover,  $\tau >0$ is the time delay, $\beta$ and $\gamma$ are nonnegative constants, $y^0, y^1,g$ are initial data given in suitable spaces, the damping coefficient $k$ belongs to $L^1_{loc}([-\tau,+\infty))$
and satisfies
\begin{equation*}
	\int_{t-\tau}^{t}|k(s)|ds\le \Lambda, \,\,\,\,\,\,\,\,\forall\; t\ge 0,
\end{equation*}
for some $\Lambda>0$.  The main feature in these problems is that the coefficient $a$ is a positive function that degenerates at $x = 0$ according to the following definitions:
\begin{Definition}
	A real function $g$ is {\it weakly degenerate at $0$}, $(WD)$ for short, if $g\in\mathcal{C}[0,1]\cap\mathcal{C}^1(0,1]$, $g(0)=0$, $g>0$ on $(0,1]$ and if
	\begin{equation}\label{sup}
		K:=\sup_{x\in (0,1]}\frac{x|g'(x)|}{g(x)},
	\end{equation}
	then $K\in (0,1)$.
\end{Definition}
\begin{Definition}
	A real function $g$ is {\it strongly degenerate at $0$}, $(SD)$ for short,  if $g\in\mathcal{C}^1[0,1]$, $g(0)=0$, $g>0$ on $(0,1]$ and in \eqref{sup} we have $K\in [1,2)$.
\end{Definition}

As a consequence, classical methods cannot be used directly to study such problems and a different approach is needed. 

Degenerate differential equations similar to \eqref{P0} have attracted a lot of attention during the last decades, mainly because they allow us to give accurate descriptions of several complex phenomena in numerous fields of science, especially in biology and engineering. Indeed, some interdisciplinary applications can be described by degenerate equations which provide an excellent instrument for the description of the properties of different processes. For parabolic degenerate problems the pioneering papers are \cite{acf}, \cite{cfr1}, \cite{cfv}, \cite{cmv}, \cite{fmANONA}, \cite{mv}, \cite{mrv} (see also \cite{fmbook} and the references therein); for hyperbolic degenerate problems  the most important paper is \cite{ACL} (see also the arxiv version of 2015), where a general degenerate function is considered (see also \cite{g}, \cite{zhang}, and the references mentioned within), and \cite{bfm} for the non divergence case (see also \cite{FM}). On the other hand, for degenerate beam problems the first results can be found in \cite{CF_Beam}, \cite{CF_Stabil} and \cite{CF_ConStab}.
However, it is important to underline that in all the previous papers there is not a delay term and the equations are linear, except for \cite{cfv} where there is a semilinear term.

Indeed, it is well known that time delay effects always exist in real systems, which may be caused, e.g., by computation of control forces. Time delay arises in many biological and physical applications and leads to a subclass of differential equations in which the derivative of the unknown function at a certain time is given in terms of the values of the function at previous times, i.e. delay differential equations. Since time delay may destroy stability (see \cite{11 datko}, \cite{12 datko1}), even if it is small, the stabilization problem of equations with delay terms becomes a very important topic.

More physical justifications together with some mathematical results on well-posedness and stability of solutions for this kind of equations are discussed, e.g., in \cite{FP}, \cite{13 NP}, for wave equations, and in \cite{Feng}, \cite{15 raposo}, \cite{17 }, \cite{18 }, \cite{16 } (see also the references therein) for beam equations. 

Here, we apply to the degenerate case a powerful tool introduced in \cite{14 NP1} and then generalized in  \cite{KP}, to analyze a class of abstract evolution equations in the presence of a  time delay when the related undelayed system is uniformly exponentially stable. Such techniques have then been refined to deal with semilinear equations with time dely \cite{PaoPi} (cf. \cite{CP} for the case of time variable time delay).

Indeed, a common feature in many delayed equations coming from applications is  the presence of a nonlinear term. There are a lot of papers dealing with second or fourth-order hyperbolic equations with suitable nonlinear sources.  For example, referring  to second-order problems,  in \cite{FP}, \cite{PaoPi} and \cite{CP}  the authors study the stability of solutions for some abstract wave equations incorporating a nonlinear term of gradient type satisfying a local Lipschitz condition. In concrete applications, such as the theory of elasticity and viscoelasticity, this nonlinear term gives rise to a nonlinear source term (see, e.g., \cite{ACS} and \cite{8 berrimi}, \cite{luo xiao} where some decay estimates are proved). 
If we keep in mind the fourth order case, then we can consider, for example, \cite{1 yang}  where the authors consider the nonlinear extensible beam equation 
\begin{equation}\label{1}
	y_{tt}+\Delta^2y-M(\norm{\nabla y}^2_{L^2(\Omega)})\Delta y-\Delta y_t+|y_t|^{m-1}y_t=|y|^{p-1}y \text{ in } \Omega\times (0,T)
\end{equation}
in order to study the evolution of the transverse deflection of an extensible beam derived from the connection mechanics. In \eqref{1} $\Omega$ is a bounded domain of $\mathbb{R}^N$, $N>1$, with smooth boundary $\partial\Omega$, $m\ge 1$, $M(s):=1+\beta s^{\gamma}$, $\gamma\ge 0$, $s\ge 0$ and $\beta, \gamma, p$ satisfy suitable hypotheses. The terms appearing in the previous equation possess a precise physical meaning; more precisely, the nonlinear term $M(\norm{\nabla y}^2_{L^2(\Omega)})\Delta y$ represents the extensibility effects on the beam, the dissipative terms $\Delta y_t$ and $|y_t|^{m-1}y_t$ represent the friction force and the nonlinear source term $|y|^{p-1}y$ represents the external load distribution. Another example is given in \cite{Radul}, where the authors make specific assumptions on the generalized source term $f(y)$; in particular, $f$ is of class $\mathcal{C}^1$, $f(0)=f'(0)=0$, $f(y)$ is monotone and convex for $y>0$, concave for $y<0$ and it satisfies suitable estimates together with $F(y)=\int_0^yf(s)ds$. In this perspective other important contributions on nonlinear beam equations with source terms are given in \cite{19 chen},  \cite{21 }, \cite{23 messaoudi}, \cite{20 }, \cite{22 } and in \cite{24 log}, where a logarithmic nonlinearity source in the right-hand side of the equation is considered. We recall also \cite{2 bainov}, where the authors give sufficient conditions for the non-existence of smooth solutions with negative initial energy,  \cite{3 park}, where existence, uniqueness, and uniform convergence of solutions are addressed,  and \cite{4 wu}, where the authors show that the local solutions blow up in a finite time with positive subcritical initial energy improving the results obtained in \cite{5 ono}.

Well posedness and stability analysis for beam problems with nonlinear source terms are treated employing different techniques. This is the case of \cite{9 ackleh}, where existence and uniqueness of weak solutions to a nonlinear beam equation are established under relaxed assumptions (locally Lipschitz plus affine domination) on the nonlinearity. In \cite{10 cavalcanti} the authors use a fixed point method and a continuity argument to establish global existence and asymptotic stability for a beam problem with nonlinear source term (see also  \cite{28}, \cite{25 pereira}, \cite{26}, \cite{27}). In \cite{luo xiao} the authors focus the attention on a second order non-autonomous hyperbolic equation in an abstract Hilbert space and apply the obtained results to Neumann or Dirichlet problems for non-autonomous, semilinear wave equations. More precisely, they consider two concrete models in a bounded domain $\Omega\subset\mathbb{R}^N, N\in\mathbb{N}$: the first one is concerned with a nonlinear source term of the type $|y(t,x)|^py(t,x)$, where $p>0$ is a positive exponent, with no further restrictions if $N=1,2$, and $p\le \frac{2}{N-2}$ if $N\ge 3$; in the second application they discuss an integro-differential wave equation with a nonlinearity of the form $(\int_{\Omega}|y(t,x)|^2dx)^{\frac{p}{2}}y(t,x), p\ge 1$. Finally, in \cite{29} the authors consider more general nonlinear source terms giving vacuum isolating phenomena of the solution and extending at the same time the results of \cite{28}.

However, in all the previous papers the studied equations are always \textit{non degenerate}, in the sense that they do not take into account a leading fourth order (or a second order) operator affected by a coefficient which degenerates somewhere in the space domain (according to the definitions given above). 

To our best knowledge, the case of degenerate wave or beam type equations with delay and suitable nonlinearities is never considered in literature.  For this reason, the aim of this work is to fill this gap; thus this is the first paper where  \textit{well posedness} and \textit{stability} for a {\it degenerate} wave or beam equation with \textit{time delay} and \textit{nonlinear source terms} are studied.
Our arguments extend to the new functional setting the method introduced in \cite{PaoPi} to deal with semilinear wave-type equations with viscoelastic damping and delay feedback. Here, the presence of the degeneracy requires a more careful analysis in the technical preliminary lemmas (see Lemma 
\ref{Lemma 3.1} and Proposition \ref{Prop E} below). Furthermore, an appropriate energy has to be defined in order to deal with the degeneracy, the delay feedback and the boundary conditions. However, with respect to \cite{PaoPi} we give a simpler proof by shortening some step by step procedure   (cf. \cite{CP} for the time-dependent time delay case).

The paper is organized as follows. In Section \ref{sezione2} we recall some useful results on the linear undelayed equation governed by a fourth order degenerate operator in non divergence form. Thanks to these results, in Section \ref{section} we deduce well posedness and stability for the delayed nonlinear problem writing it in an abstract way and using the Duhamel formula. In Section \ref{Sez 4} some illustrative examples are given. Finally, in Section \ref{Sez 5} we extend the  results obtained in Section \ref{section} to problems governed by a degenerate fourth order operator in divergence form or by a degenerate second order operator in divergence or in non divergence form.

	\section{The linear undelayed problem}\label{sezione2}
	In this section we consider  the undelayed problem
	\begin{equation}\label{(P_undelay)}
		\begin{cases}
			y_{tt}(t,x)+a(x)y_{xxxx}(t,x)=0, &(t,x)\in Q,\\
			y(t,0)=0,\,\,y_x(t,0)=0, &t>0,\\
			\beta y(t,1)-y_{xxx}(t,1)+y_t(t,1)=0, &t >0,\\
			\gamma y_x(t,1)+y_{xx}(t,1)+y_{tx}(t,1)=0, &t >0,\\
			y(0,x)=y^0(x),\,\,y_t(0,x)=y^1(x),&x\in(0,1),
		\end{cases}
	\end{equation}
	where	$Q:=(0,+\infty) \times (0,1)$ and $\beta, \gamma \ge 0$. In particular, following \cite{CF_Stabil}, we present some functional spaces and some results crucial for the following.
		
	As in  \cite{CF}, \cite{CF_Neumann}, \cite{CF_Wentzell} or \cite{CF_Beam}, let us consider the following weighted Hilbert spaces with the related inner products:
	\begin{equation*}
		L^2_{\frac{1}{a}}(0, 1):=\biggl \{u\in L^2(0, 1):\int_{0}^{1}\frac{u^2}{a}\,dx<+\infty \biggr \},
		\end{equation*}
		\begin{equation*}	\left\langle u,v\right\rangle_{L^2_{\frac{1}{a}}(0,1)}:=\int_0^1\frac{uv}{a}dx,
	\end{equation*}
	for every $u,v \in L^2_{\frac{1}{a}}(0,1)$, and
	\begin{equation*}
		H^i_{\frac{1}{a}}(0, 1):= L^2_{\frac{1}{a}}(0, 1)\cap H^i(0, 1),
			\end{equation*}
		\begin{equation*} \langle u,v\rangle_{H^i_{\frac{1}{a}}(0,1)}:=\langle u,v\rangle_{L^2_{\frac{1}{a}}(0,1)}+\sum_{k=1}^{i}\langle u^{(k)}, v^{(k)}\rangle_{L^2(0,1)},
	\end{equation*}
	for every $u,v \in 	H^i_{\frac{1}{a}}(0, 1)$, $i=1,2$. Obviously, the previous inner products induce the related respective norms
	\begin{equation*}
		\norm{u}^2_{L^2_{\frac{1}{a}}(0, 1)}:= \int_{0}^{1}\frac{u^2}{a}\,dx,\quad\,\,\,\,\,\,\,\forall\,u\in L^2_{\frac{1}{a}}(0, 1)
	\end{equation*}
	and
	\begin{equation*}
		\norm{u}_{H^i_{\frac{1}{a}}(0, 1)}^2:=\norm{u}^2_{L^2_{\frac{1}{a}}(0, 1)} + \sum_{k=1}^{i}\Vert u^{(k)}\Vert^2_{L^2(0, 1)}, \quad\,\,\,\,\,\,\,\forall\,u\in H^i_{\frac{1}{a}}(0, 1),
	\end{equation*}
	$i=1,2$.  	Observe that  $\|\cdot\|_{H^2_{\frac{1}{a}}(0, 1)}$ is equivalent to $\|\cdot\|_2$ in $H^2_{\frac{1}{a}}(0, 1)$, where
	\[
	\| u \|_2^2:= \|u\|_{L^2_{\frac{1}{a}}(0, 1)}^2+ \|u''\|_{L^2(0,1)}^2,
	\]
	for all $u \in H^2_{\frac{1}{a}}(0, 1)$
	(see, e.g., \cite{CF_Wentzell}).  In addition to the previous Hilbert spaces, we introduce the following ones:
	\begin{equation*}
		H^1_{\frac{1}{a},0}(0, 1):= \biggl \{u\in H^1_{\frac{1}{a}}(0, 1):u(0)=0 \biggr \},	\end{equation*}
	\begin{equation*}
		H^2_{\frac{1}{a},0}(0, 1):= \biggl \{u\in H^1_{\frac{1}{a},0}(0, 1)\cap H^2(0,1):u'(0)=0 \biggr \},
	\end{equation*}
	with norms $\|\cdot\|_{H^i_{\frac{1}{a}}(0, 1)}$, $i=1,2$. Assuming  that $a$ is (WD) or (SD), one can prove that  $\|\cdot\|_{H^2_{\frac{1}{a}}(0, 1)}$ and $\|\cdot\|_2$ are equivalent to the next one
	\begin{equation}\label{stimanorme}
	\|u\|_{2, \sim }:= \|u''\|_{L^2(0,1)}, \quad \forall\, u\in H^2_{\frac{1}{a},0}(0, 1),
	\end{equation}
	 (see \cite{CF_Stabil}).  Indeed, as proved in \cite[Proposition 2.1]{CF_Stabil}, 
%	\[
%	\|u\|_1^2 \le (C_{HP}+1) \|u\|^2_{1,\sim}, \quad \forall \; u \in H^1_{\frac{1}{a},0}(0, 1),
%	\]
%	and
	\begin{equation}\label{stimanormeequi}
	\|u\|_2^2 \le (4C_{HP}+1) \|u\|^2_{2,\sim}, \quad \forall \; u \in H^2_{\frac{1}{a},0}(0, 1),
	\end{equation}
	where $C_{HP}$ is the best constant of the Hardy-Poincar\'e inequality introduced in \cite[Proposition 2.6]{cfr1}. 	\noindent	Finally, we introduce the last important Hilbert space:
	\begin{equation*}
		\mathcal{H}_0:=H^2_{\frac{1}{a},0}(0,1)\times L^2_{\frac{1}{a}}(0,1),
	\end{equation*}
	endowed with inner product
	\begin{equation*}
		\langle (u,v),(\tilde{u},\tilde{v})\rangle_{\mathcal{H}_0}:=\int_{0}^{1}u''\tilde{u}''dx+\int_{0}^{1}\frac{v\tilde{v}}{a}dx+\beta u(1)\tilde{u}(1)+\gamma u'(1)\tilde{u}'(1)
	\end{equation*}
	and with  norm
	\begin{equation*}
		\|(u,v)\|^2_{\mathcal{H}_0}:=\int_{0}^{1}(u'')^2dx+\int_{0}^{1}\frac{v^2}{a}dx+\beta u^2(1)+\gamma (u'(1))^2
	\end{equation*}
	for every $(u,v), \;(\tilde{u},\tilde{v})\in\mathcal{H}_0$. 
On $\mathcal H_0$ we can define the matrix operator 
$\mathcal{A}_{nd}:D(\mathcal{A}_{nd})\subset\mathcal{H}_0\to \mathcal{H}_0$ given by
	\begin{equation}\label{matrix A}
		\mathcal{A}_{nd}:=\begin{pmatrix}
			0 & Id \\
			-A_{nd} & 0
		\end{pmatrix}
	\end{equation}
	with domain
	\begin{equation*}
		\begin{aligned}
			D(\mathcal{A}_{nd}):=\{(u,v)\in D(A_{nd})\times H^2_{\frac{1}{a},0}(0,1): &\beta u(1)-u'''(1)+v(1)=0,\\ &\gamma u'(1)+u''(1)+v'(1)=0 \},
		\end{aligned}
	\end{equation*}
where
	\begin{equation*}
		A_{nd}u:=au'''' \quad \text{for all } u \in
		D(A_{nd}):=\left\{u\in H^2_{\frac{1}{a},0}(0, 1): au''''\in L^2_{\frac{1}{a}}(0, 1) \right\}.
	\end{equation*}	Using the previous spaces, one can prove the following Gauss Green formula	\begin{equation}\label{GF1}
		\int_0^1 u''''v dx= u'''(1)v(1)-u''(1)v'(1) +\int_0^1 u''v''dx
	\end{equation}
	for all $(u,v) \in D(A_{nd}) \times H^2_{\frac{1}{a},0}(0, 1)$.
	Thanks to \eqref{GF1} one can prove that $(\mathcal A_{nd}, D(\mathcal A_{nd}))$ is non positive with dense domain and generates a contraction semigroup  $(S(t))_{t \ge 0}$ (see \cite{CF_Stabil}) as soon as
		 $a$ is (WD) or (SD). Moreover, as in \cite{CF_Beam} or \cite{CF_Stabil}, one can prove the following well posedness result:
	\begin{Theorem}\label{Theorem regol}
		Assume $a$ (WD) or (SD).
		If $(y^0,y^1)\in\mathcal{H}_0$, then there exists a unique mild solution
		\begin{equation*}
			y\in \mathcal{C}^1([0,+\infty);L^2_{\frac{1}{a}}(0,1))\cap \mathcal{C}([0,+\infty);H^2_{\frac{1}{a},0}(0,1))
		\end{equation*}
		of \eqref{(P_undelay)} which depends continuously on the initial data $(y^0,y^1)\in \mathcal{H}_0$. Moreover, if $(y^0,y^1)\in D(\mathcal{A}_{nd})$, then the solution $y$ is classical, in the sense that
	\begin{equation*}
		y\in \mathcal{C}^2([0,+\infty);L^2_{\frac{1}{a}}(0,1))\cap \mathcal{C}^1([0,+\infty);H^2_{\frac{1}{a},0}(0,1))\cap \mathcal{C}([0,+\infty);D(A_{nd}))
	\end{equation*}
		and the equation of \eqref{(P_undelay)} holds for all $t\ge 0$.
	\end{Theorem}
		
	Hence, thanks to the previous result, if  $a$ is (WD) or (SD), then there exists a unique mild solution $y$  of (\ref{(P_undelay)}) and we can define its energy as
		\begin{equation*}
			\mathcal{E}_y(t):=\frac{1}{2}\int_0^1 \Biggl (\frac{y^2_t(t,x)}{a(x)}+y^2_{xx}(t,x) \Biggr )dx+\frac{\beta}{2}y^2(t,1)+\frac{\gamma}{2}y_x^2(t,1),\quad\,\,\,\,\,\,\forall\;t\ge 0,
		\end{equation*}
		where $\beta,\gamma \ge 0$. In addition, if $y$ is classical, then the energy is non increasing and		\[
		\frac{d\mathcal{E}_y(t)}{dt}=-y^2_t(t,1)-y^2_{tx}(t,1),\quad\,\,\,\,\,\,\,\forall\;t\ge 0.
		\]
In particular the following stability result holds.
	
	\begin{Theorem}\label{teoremaprincipale}\cite[Theorem 3.2]{CF_Stabil}
		Assume $a$ (WD) or (SD) and let $y$ be a mild solution of \eqref{(P_undelay)}. Then, there exists a suitable constant $T_0>0$ such that 		\begin{equation*}
			\mathcal{E}_y(t)\le \mathcal{E}_y(0)e^{1-\frac{t}{T_0}},
		\end{equation*}
		for all $t\ge T_0$.	\end{Theorem}
	We underline that in the previous result the condition $K <2$ is only a technical hypothesis. Indeed, as written in \cite[Section 4]{CF_Stabil}, the stability for \eqref{(P_undelay)} when $K \ge 2$ is still an {\it open problem}.
	Moreover, under the conditions provided in the previous theorem, the exponential decay of solutions for \eqref{(P_undelay)} is uniform. In particular, as a consequence of Theorem \ref{teoremaprincipale}, we know that the  $\mathcal{C}_0$-semigroup generated by $(\mathcal{A}_{nd}, D(\mathcal A_{nd}))$, $(S(t))_{t\ge 0}$,  is exponentially stable, i.e. there exist $M,\omega >0$ such that
\begin{equation}\label{stima S(t)}
	\norm{S(t)}_{\mathcal{L}(\mathcal{H}_0)}\le Me^{-\omega t},\,\,\,\,\,\,\,\,\,\forall\; t\ge 0
\end{equation}
(see, for example, \cite[Theorems 2.1 and 2.2]{CF_Stabil}).

\section{The delayed equation}\label{section}
	In this section we analyze well posedness and stability for
		\begin{equation}\label{(P)}
		\begin{cases}
			y_{tt}(t,x)+a(x)y_{xxxx}(t,x)+k(t)BB^*y_t(t-\tau,x)=f(y(t,x)), &(t,x)\in Q,\\
			y(t,0)=0,\,\,y_x(t,0)=0, &t>0,\\
			\beta y(t,1)-y_{xxx}(t,1)+y_t(t,1)=0, &t >0,\\
			\gamma y_x(t,1)+y_{xx}(t,1)+y_{tx}(t,1)=0, &t >0,\\
			y(0,x)=y^0(x),\,\,y_t(0,x)=y^1(x),&x\in(0,1),\\
			B^*y_t(s,x)=g(s), &s\in [-\tau,0],
		\end{cases}
	\end{equation}
	where  $\tau >0$ is the time delay, $g$ is defined in $[-\tau,0]$ with values on a real  Hilbert space $H$, $B: H \rightarrow L^2_{\frac{1}{a}}(0,1)$ is a bounded linear operator with adjoint $B^*$ and $Q$, $\beta, \gamma$ are as in the previous section.
	
Now,
defining  $v(t,x):=y_t(t,x)$,  $Y^0(x):= \begin{pmatrix} y^0(x)\\ y^1(x)\end{pmatrix}$, $Y(t,x):=\begin{pmatrix} y(t,x)\\ v(t,x)\end{pmatrix} $ and using the  operators
\begin{equation*}
\psi(s):=\begin{pmatrix}
	0  \\
	Bg(s)\end{pmatrix}, \;  \mathcal{B}Y(t):=\begin{pmatrix}
	0  \\
	BB^*v(t)
\end{pmatrix}, \; \mathcal{F}(Y(t)):=\begin{pmatrix}
	0  \\
	f(y(t,x))
\end{pmatrix},
\end{equation*}
and $(\mathcal A_{nd}, D(\mathcal A_{nd}))$ defined in \eqref{matrix A},
 \eqref{(P)} can be formulated in the following abstract form
\begin{equation}\label{(P_abstract)}
	\begin{cases}
		\dot{Y}(t)=\mathcal{A}_{nd}Y(t)-k(t)\mathcal{B}Y(t-\tau)+\mathcal{F}(Y(t)), &(t,x) \in Q,\\
		Y(0)=Y^0,& x \in (0,1),\\
		\mathcal{B}Y(s)=\psi(s), &s \in [-\tau,0].
	\end{cases}
\end{equation}
Observe that, if $Y^0\in\mathcal{H}_0$, then the following Duhamel formula holds:
\begin{equation}\label{duhamel}
	Y(t)=S(t)Y^0+\int_0^tS(t-s)k(s)\mathcal{B}Y(s-\tau)ds+\int_0^tS(t-s)\mathcal{F}(Y(s))ds.
\end{equation}Moreover, setting
\begin{equation}\label{b}
b:=	\norm{B}_{\mathcal{L}(H,L^2_{\frac{1}{a}}(0,1))}=\norm{B^*}_{\mathcal{L}(L^2_{\frac{1}{a}}(0,1),H)},
\end{equation}we have
\begin{equation*}
	\norm{\mathcal{B}}_{\mathcal{L}(\mathcal{H}_0)}=b^2,
\end{equation*}
by \cite{ANP} and \cite{TucWei}.

In order to treat  \eqref{(P_abstract)}, we make the following assumptions on $k$ and $f$:
	\begin{Assumptions}\label{ipo k}
		The function $k: [-\tau, +\infty) \rightarrow \R$ belongs to
$L^1_{loc}([-\tau,+\infty))$ and there exists  $\Lambda>0$ such that
		\begin{equation*}
			\int_{t-\tau}^{t}|k(s)|ds\le \Lambda, \,\,\,\,\,\,\,\,\forall\; t\ge 0.
		\end{equation*}
	\end{Assumptions}
	
	\begin{Assumptions}\label{ipo f}
		Let $f:H^2_{\frac 1 a, 0}(0,1)\to L^2_{\frac 1 a}(0,1)$ be a continuous function such that		\begin{enumerate}
		\item $f(0)=0$;
			\item for all $r>0$ there exists a constant $L(r)>0$ such that, for all $u,v\in H^2_{\frac 1 a, 0}(0,1)$ satisfying $\norm{u''}_{L^2(0,1)}\le r$ and $\norm{v''}_{L^2(0,1)}\le r$, one has
			\begin{equation*}
				\norm{f(u)-f(v)}_{L^2_{\frac{1}{a}}(0,1)}\le L(r)\norm{u''-v''}_{L^2(0,1)};
			\end{equation*}
			\item there exists a strictly increasing continuous function $h: \R_+ \rightarrow \R_+ $ such that 
			\begin{equation}\label{condizione}
				\langle f(u), u\rangle_{L^2_{\frac{1}{a}}(0,1)}\le h(\norm{u''}_{L^2(0,1)})\norm{u''}^2_{L^2(0,1)}
			\end{equation}
		for all $u\in H^2_{\frac 1 a, 0}(0,1)$.
		\end{enumerate}
	\end{Assumptions}

\begin{Assumptions}\label{ipo Mbe+WP}
	Suppose that:
	\begin{enumerate}
		\item for any $t>0$ \begin{equation}\label{ipo Mbe}
			Mb^2e^{\omega \tau}\int_0^t|k(s+\tau)|ds\le \alpha +\omega't
		\end{equation}
		for suitable constants $\alpha\ge 0$ and $\omega'\in [0,\omega)$, where $M$, $\omega$  and $b$ are the constants in \eqref{stima S(t)} and \eqref{b}, respectively;
		\item there exist $T, \rho , C_\rho>0$, with $L(C_\rho)<\frac{\omega - \omega'}{M}$, such that if $Y^0\in\mathcal{H}_0$ and $g: [-\tau,0] \rightarrow H$ satisfy
		\begin{equation}\label{ipo WP}
			\norm{Y^0}^2_{\mathcal{H}_0}+\int_0^\tau |k(s)|\cdot \norm{g(s-\tau)}^2_Hds<\rho^2,
		\end{equation}
	then \eqref{(P_abstract)} has a unique solution $Y\in\mathcal{C}([0, T);\mathcal{H}_0)$ satisfying $\norm{Y(t)}_{\mathcal{H}_0}\le C_\rho$ for all $t\in [0, T)$.
	\end{enumerate}
\end{Assumptions}
In particular, Hypothesis \ref{ipo k} is crucial to prove the existence of a local unique solution for \eqref{(P_abstract)} (see Proposition \ref{th esist loc}); while Hypotheses \ref{ipo k}, \ref{ipo f} are needed to prove that \eqref{(P_abstract)} satisfies Hypothesis \ref{ipo Mbe+WP}.2 and, if the initial data are sufficiently small,  the corresponding solutions exist and decay  exponentially  in $(0, +\infty)$ (see Theorem \ref{global}).
On the other hand, the exponential stability is proved thanks to Hypotheses \ref{ipo f} and \ref{ipo Mbe+WP} (see Theorem \ref{Th Duhamel}).
Moreover, observe that Hypothesis \ref{ipo Mbe+WP}.1 is satisfied, in particular, if  $k\in L^1[0,+\infty)$ or $k\in L^{\infty}[0,+\infty)$ and $\norm{k}_{L^\infty(0,1)}$ is smaller than a suitable constant depending on $M,\omega, b$ and $\tau$.

Furthermore, thanks to Hypothesis \ref{ipo f}, $\mathcal{F}(0)=0$ and for any $r>0$ there exists a constant $L(r)>0$ such that
\begin{equation*}
	\norm{\mathcal{F}(Y)-\mathcal{F}(Z)}_{\mathcal{H}_0}\le L(r)\norm{Y-Z}_{\mathcal{H}_0}
\end{equation*} 
whenever $\norm{Y}_{\mathcal{H}_0}\le r$ and $\norm{Z}_{\mathcal{H}_0}\le r$. In particular, 

\[
	\norm{\mathcal{F}(Y)}_{\mathcal{H}_0}\le L(r)\norm{Y}_{\mathcal{H}_0}.
\]

\vspace{0,3cm}

To conclude this section, define
\begin{equation}\label{Fgrande}
F(y):=\int_0^y f(s) ds, \quad y\in H^2_{\frac 1 a, 0}(0,1)
\end{equation}
and
observe that it is possible to prove the following estimate on the nonlinear term $\displaystyle\int_0^1\frac{F(y(x))}{a(x)}dx$ thanks to Hypothesis \ref{ipo f}:
\begin{Lemma}\label{Lemma 3.1}
	Assume Hypothesis \ref{ipo f} and $a$ (WD) or (SD). Then
	\begin{equation}\label{stima nonlin}
		\Biggl |\int_0^1\frac{F(y(x))}{a(x)}dx\Biggr |\le \frac{1}{2}h(\norm{y''}_{L^2(0,1)})\norm{y''}^2_{L^2(0,1)},
	\end{equation}
for all $y\in H^2_{\frac 1 a, 0}(0,1)$.
\end{Lemma}
\begin{proof}
	Fix $y\in H^2_{\frac 1 a, 0}(0,1)$. Observing that $\frac{d}{ds}F(sy)=F'(sy)y=f(sy)y$, we have
	\begin{equation*}
		\begin{aligned}
			\int_0^1\frac{F(y(x))}{a(x)}dx&=\int_0^1\frac{1}{a(x)}\int_0^1f(sy(x))y(x)\,ds\,dx=\int_0^1\langle f(sy), sy\rangle_{L^2_{\frac{1}{a}}(0,1)}\frac{ds}{s}.
		\end{aligned}
	\end{equation*}
Thus, by \eqref{condizione},
\begin{equation*}
	\begin{aligned}
		\Biggl |\int_0^1\frac{F(y(x))}{a(x)}dx\Biggr |&\le \int_0^1h(\norm{sy''}_{L^2(0,1)})s^2\norm{y''}^2_{L^2(0,1)}\frac{ds}{s}	\\
		&\le \frac{1}{2}h(\norm{y''}_{L^2(0,1)})\norm{y''}^2_{L^2(0,1)}.
	\end{aligned}
\end{equation*}
\end{proof}

\subsection{Exponential stability}
	Under the well posedness assumption \eqref{ipo WP}, in this subsection we will give the exponential decay result for problem \eqref{(P_abstract)}. 
As a first step, we give an abstract stability result. This is similar to \cite[Theorem 2.1]{PaoPi}, but here we give the proof for the reader's convenience.

	\begin{Theorem}\label{Th Duhamel}
		Assume Hypotheses \ref{ipo f} and \ref{ipo Mbe+WP}, $a$ (WD) or (SD) and consider  the initial data $(Y^0,g)$ satisfying \eqref{ipo WP}. Then every solution $Y$ of \eqref{(P_abstract)} is such that
		\begin{equation}\label{exp decay}
			\norm{Y(t)}_{\mathcal{H}_0}\le Me^\alpha \Biggl (	\norm{Y^0}_{\mathcal{H}_0}+\int_0^\tau e^{\omega s}|k(s)|\cdot \norm{\psi(s-\tau)}_{\mathcal{H}_0}ds \Biggr )e^{-(\omega -\omega'-ML(C_\rho))t},
		\end{equation}
		for any $t\in [0, T)$.
	\end{Theorem}
\begin{proof}
	Since $\norm{\mathcal{F}(Y(t))}_{\mathcal{H}_0}\le L(C_\rho)\norm{Y(t)}_{\mathcal{H}_0}$ for every $t\in [0, T)$, by \eqref{duhamel} we have
	\begin{equation*}
		\begin{aligned}
			\norm{Y(t)}_{\mathcal{H}_0}&\le Me^{-\omega t}\norm{Y^0}_{\mathcal{H}_0}+Me^{-\omega t}\int_0^te^{\omega s}|k(s)|\cdot \norm{\mathcal{B}Y(s-\tau)}_{\mathcal{H}_0}ds\\
			&+ML(C_\rho)e^{-\omega t}\int_0^te^{\omega s}\norm{Y(s)}_{\mathcal{H}_0}ds,
		\end{aligned}
	\end{equation*}
	where we recall that $M,\omega$ and $b$ are the parameters appearing in \eqref{stima S(t)} and \eqref{b}, respectively. In particular, we obtain
	\begin{equation}\label{norma Y(t)}
		\begin{aligned}
			\norm{Y(t)}_{\mathcal{H}_0}&\le Me^{-\omega t}\norm{Y^0}_{\mathcal{H}_0}+Me^{-\omega t}\int_0^\tau e^{\omega s}|k(s)|\cdot \norm{\psi(s-\tau)}_{\mathcal{H}_0}ds\\
			&+Me^{-\omega t}\int_\tau^t e^{\omega s}b^2|k(s)|\cdot \norm{Y(s-\tau)}_{\mathcal{H}_0}ds\\
			&+ML(C_\rho)e^{-\omega t}\int_0^te^{\omega s}\norm{Y(s)}_{\mathcal{H}_0}ds,
		\end{aligned}
	\end{equation}
	if $\tau \le t$ and
	\begin{equation}\label{norma Y(t)1}
		\begin{aligned}
			\norm{Y(t)}_{\mathcal{H}_0}&\le Me^{-\omega t}\norm{Y^0}_{\mathcal{H}_0}+Me^{-\omega t}\int_0^\tau e^{\omega s}|k(s)|\cdot \norm{\psi(s-\tau)}_{\mathcal{H}_0}ds\\
			&+ML(C_\rho)e^{-\omega t}\int_0^te^{\omega s}\norm{Y(s)}_{\mathcal{H}_0}ds,
		\end{aligned}
	\end{equation}
	if $t< \tau$.
Setting $z:=s-\tau$ in the second integral of \eqref{norma Y(t)} or \eqref{norma Y(t)1} and multiplying the previous inequality by $e^{\omega t}$, we get
\begin{equation*}
	\begin{aligned}
		e^{\omega t}\norm{Y(t)}_{\mathcal{H}_0}&\le M\norm{Y^0}_{\mathcal{H}_0}+M\int_0^\tau e^{\omega s}|k(s)|\cdot \norm{\psi(s-\tau)}_{\mathcal{H}_0}ds\\
		&+Mb^2e^{\omega \tau}\int_0^t e^{\omega z}|k(z+\tau)|\cdot \norm{Y(z)}_{\mathcal{H}_0}dz\\
		&+ML(C_\rho)\int_0^te^{\omega s}\norm{Y(s)}_{\mathcal{H}_0}ds,
	\end{aligned}
\end{equation*}
if $\tau \le t$ and 
\begin{equation*}
	\begin{aligned}
		e^{\omega t}\norm{Y(t)}_{\mathcal{H}_0}&\le M\norm{Y^0}_{\mathcal{H}_0}+M\int_0^\tau e^{\omega s}|k(s)|\cdot \norm{\psi(s-\tau)}_{\mathcal{H}_0}ds\\
		&+ML(C_\rho)\int_0^te^{\omega s}\norm{Y(s)}_{\mathcal{H}_0}ds,
	\end{aligned}
\end{equation*}
if $t < \tau$.
Now, let us denote $M_0:=M\norm{Y^0}_{\mathcal{H}_0}+M\int_0^\tau e^{\omega s}|k(s)|\cdot \norm{\psi(s-\tau)}_{\mathcal{H}_0}ds$. Thus, by Gronwall's Lemma and \eqref{ipo Mbe}, we have 
	\begin{equation*}
		\norm{Y(t)}_{\mathcal{H}_0}\le M_0e^{Mb^2	e^{\omega \tau}\int_0^t|k(s+\tau)|ds+ML(C_\rho)t-\omega t}\le M_0e^\alpha e^{[ML(C_\rho)-(\omega -\omega')]t},
	\end{equation*}
	if $\tau \le t$ and
	\begin{equation*}
		\norm{Y(t)}_{\mathcal{H}_0}\le M_0e^{ML(C_\rho)t-\omega t}\le M_0e^{[ML(C_\rho)-\omega]t},
	\end{equation*}
if $t < \tau$. In any case \eqref{exp decay}  holds.
\end{proof}
For the next step, we define the appropriate  energy functional.
		\begin{Definition}
	Let $y$ be a mild solution of (\ref{(P)}) and define its energy  as
	\begin{equation*}
	\begin{aligned}
		E_y(t)&:=\frac{1}{2}\int_0^1 \Biggl (\frac{y^2_t(t,x)}{a(x)}+y^2_{xx}(t,x) \Biggr )dx+\frac{\beta}{2}y^2(t,1)+\frac{\gamma}{2}y_x^2(t,1)\\
		&-\int_0^1\frac{F(y(t,x))}{a(x)} dx+\frac{1}{2}\int_{t-\tau}^t |k(s+\tau)|\cdot \norm{B^*y_t(s)}^2_{H}ds,\quad\,\,\,\,\,\,\forall\;t\ge 0,
	\end{aligned}
	\end{equation*}
where $F$ is defined in \eqref{Fgrande}.
\end{Definition}

The following result holds.

\begin{Theorem}\label{th E}
Assume  Hypothesis \ref{ipo Mbe+WP}.2, $a$ (WD) or (SD)	and let $y$ be a mild solution of \eqref{(P)} defined on $[0, T).$ If $E_y(t)\ge\frac{1}{4}\norm{y_t(t)}^2_{L^2_{\frac{1}{a}}(0,1)}$ for any $t\in [0, T),$ then
	\begin{equation*}
		E_y(t)\le C(t)E_y(0),\,\,\,\,\,\,\,\,\,\,\forall\; t\in [0, T),
	\end{equation*}
where 
\begin{equation}\label{function C}
C(t):=e^{2\int_0^tb^2(|k(s)|+|k(s+\tau)|)ds}.
\end{equation}
\end{Theorem}
	\begin{proof}
	Let $y$ be a mild solution of \eqref{(P)}. Differentiating formally $E_y$ with respect to $t$, using \eqref{GF1} and the boundary conditions, we obtain
		\begin{equation*}
			\begin{aligned}
				\frac{dE_y(t)}{dt}&= \int_0^1\Big (
\frac{y_t(t,x)y_{tt}(t,x)}{a(x)} +y_{xx}(t,x)y_{xxt}(t,x)\Big ) dx\\
&+\beta y(t,1)y_t(t,1)+\gamma y_x(t,1)y_{tx}(t,1)\\
			&-\int_0^1\frac{f(y(t,x))y_t(t,x)}{a(x)} dx+\frac{1}{2}|k(t+\tau)|\cdot \norm{B^*y_t(s)}^2_{H}\\&
			-\frac{1}{2}|k(t)|\cdot \norm{B^*y_t(t-\tau)}^2_{H}\\
			&= \int_0^1\Big (
			\frac{y_t(t,x)y_{tt}(t,x)}{a(x)} +y_{xxxx}(t,x)y_{t}(t,x)\Big ) dx\\&-y_{xxx}(t,1)y_t(t,1)+y_{xx}(t,1)y_{tx}(t,1)\\
			&+y_t(t,1)[y_{xxx}(t,1)-y_t(t,1)]+y_{tx}(t,1)[-y_{xx}(t,1)-y_{tx}(t,1)]\\
			&-\int_0^1\frac{f(y(t,x))y_t(t,x)}{a(x)} dx+\frac{1}{2}|k(t+\tau)|\cdot \norm{B^*y_t(s)}^2_{H}\\&-\frac{1}{2}|k(t)|\cdot \norm{B^*y_t(t-\tau)}^2_{H}\\
			&= \int_0^1\Big (
			\frac{y_t(t,x)y_{tt}(t,x)}{a(x)} +y_{xxxx}(t,x)y_{t}(t,x)\Big ) dx -y_t^2(t,1)-y_{tx}^2(t,1)\\
			&-\int_0^1\frac{f(y(t,x))y_t(t,x)}{a(x)} dx+\frac{1}{2}|k(t+\tau)|\cdot \norm{B^*y_t(s)}^2_{H}\\&-\frac{1}{2}|k(t)|\cdot \norm{B^*y_t(t-\tau)}^2_{H}.
			\end{aligned}
		\end{equation*}
		Using the differential equation in \eqref{(P)}, we have
		\begin{equation*}
			\begin{aligned}
				\frac{dE_y(t)}{dt}&
				= -y_t^2(t,1)-y_{tx}^2(t,1) -k(t)\left\langle BB^*y_t(t-\tau), y_t(t)\right\rangle_{L^2_{\frac 1a}(0,1)}\\
				&+\frac{1}{2}|k(t+\tau)|\cdot \norm{B^*y_t(t)}^2_H-\frac{1}{2}|k(t)|\cdot \norm{B^*y_t(t-\tau)}^2_H\\\
				&=-y_t^2(t,1)-y_{tx}^2(t,1) -k(t)\left\langle B^*y_t(t),B^*y_t(t-\tau)\right\rangle_H\\
				&+\frac{1}{2}|k(t+\tau)|\cdot \norm{B^*y_t(t)}^2_H-\frac{1}{2}|k(t)|\cdot \norm{B^*y_t(t-\tau)}^2_H.
			\end{aligned}
		\end{equation*}
	By the Cauchy inequality, we get
	\begin{equation*}
		\begin{aligned}
			\frac{dE_y(t)}{dt}&\le \frac{1}{2}(|k(t+\tau)|+|k(t)|) \norm{B^*y_t(t)}^2_H\\
			&\le 2b^2(|k(t+\tau)|+|k(t)|)\frac{1}{4}\norm{y_t(t)}^2_{L^2_{\frac{1}{a}}(0,1)}.
		\end{aligned}
	\end{equation*}
	Using the fact that $E_y(t)\ge\frac{1}{4}\norm{y_t(t)}^2_{L^2_{\frac{1}{a}}(0,1)}$ for all $t \in [0,T)$, we get
	\begin{equation*}
		\frac{dE_y(t)}{dt}\le 2b^2(|k(t+\tau)|+|k(t)|)E_y(t)
	\end{equation*}
	and the thesis follows using the Gronwall Lemma.
	\end{proof}
		\subsection{The well posedness assumption}
	In this subsection we prove the well posedness assumption, i.e. Hypothesis \ref{ipo Mbe+WP}.2, for \eqref{(P_abstract)}. To this aim the following two propositions are crucial.
	
	\begin{Proposition}\label{th esist loc}
	Assume Hypothesis \ref{ipo k} and $a$ (WD) or (SD). Let us consider \eqref{(P_abstract)} with initial data $Y^0\in\mathcal{H}_0$ and $\psi\in \mathcal{C}([-\tau,0];\mathcal{H}_0)$. Then there exists a unique continuous local solution.
	\end{Proposition}
\begin{proof}
	It is sufficient to observe that if $t \in [0,\tau]$, then $t-\tau\in [-\tau,0]$. Thus, \eqref{(P_abstract)} can be formulated as an undelayed problem in the interval $[0,\tau]$:
	\begin{equation*}
		\begin{cases}
			\dot{Y}(t)=\mathcal{A}_{nd}Y(t)-k(t)\psi(t-\tau)+\mathcal{F}(Y(t)),&t\in (0,\tau),\\
			Y(0)=Y^0.
		\end{cases}
	\end{equation*}
Then, from the standard theory for inhomogeneous evolution problems (see \cite[Chapter 6, Theorem 1.4]{Pazy} or  \cite{PZ}) we have that there exists a unique solution of \eqref{(P_abstract)} on $[0,\delta)$, for some  $\delta\le\tau$.
\end{proof}

	\begin{Proposition}\label{Prop E}
	Assume Hypothesis \ref{ipo f}, $a$ (WD) or (SD) and consider \eqref{(P_abstract)} with initial data $Y^0\in\mathcal{H}_0$ and $\psi\in \mathcal{C}([-\tau,0];\mathcal{H}_0)$. Take $T>0$ and let $Y$ be a non trivial solution of \eqref{(P_abstract)} defined on $[0,\delta),$ with $\delta\le T$. The following statements hold: 
	\begin{enumerate}
		\item if $h\bigl (\norm{(y^0)''}_{L^2(0,1)}\bigr )<\frac{1}{2}$, then $E_y(0)>0$;
		\item if $h\bigl (\norm{(y^0)''}_{L^2(0,1)}\bigr )<\frac{1}{2}$ and $h(2\sqrt{C(T)E_y(0)})<\frac{1}{2}$, then
		\begin{equation}\label{stima_dal_basso}
\begin{array}{l}
			\displaystyle{E_y(t)>\frac{1}{4}\norm{y_t(t)}^2_{L^2_{\frac{1}{a}}(0,1)}+\frac{1}{4}\norm{y_{xx}(t)}^2_{L^2(0,1)}+\frac{\beta}{4}y^2(t,1)+\frac{\gamma}{4}y_x^2(t,1)}
\\\displaystyle{
\hspace{5,5 cm}+\frac{1}{4}\int_{t-\tau}^t |k(s+\tau)|\cdot \norm{B^*y_t(s)}^2_{H}ds}
\end{array}
		\end{equation}
		for all $t\in [0,\delta)$, being $C(\cdot)$ the function defined in \eqref{function C}. In particular,
		\begin{equation*}
			E_y(t)>\frac{1}{4}\norm{Y(t)}^2_{\mathcal{H}_0},\,\,\,\,\,\,\,\,\,\,\forall\; t\in [0,\delta).
		\end{equation*}
	\end{enumerate}
    \end{Proposition}
\begin{proof}
\underline{Claim $1.$} Let us consider a non trivial solution $Y$. By \eqref{stima nonlin} and the assumption $h\bigl (\norm{(y^0)''}_{L^2(0,1)}\bigr )<\frac{1}{2}$, we note that
\begin{equation*}
	\Biggl |\int_0^1\frac{F(y^0(x))}{a(x)}dx\Biggr |\le \frac{1}{2}h(\norm{(y^0)''}_{L^2(0,1)})\norm{(y^0)''}^2_{L^2(0,1)}<\frac{1}{4}\norm{(y^0)''}^2_{L^2(0,1)}.
\end{equation*}
As a consequence,
\begin{equation*}
	\begin{aligned}
		E_y(0)&=\frac{1}{2}\norm{y^1}^2_{L^2_\frac{1}{a}(0,1)}+\frac{1}{2}\norm{(y^0)''}^2_{L^2(0,1)}+\frac{\beta}{2}y^2(0,1)+\frac{\gamma}{2}y_x^2(0,1)\\
		&-\int_0^1\frac{F(y^0(x))}{a(x)} dx+\frac{1}{2}\int_{-\tau}^0 |k(s+\tau)|\cdot \norm{B^*y_t(s)}^2_{H}ds\\
		&>\frac{1}{4}\norm{y^1}^2_{L^2_\frac{1}{a}(0,1)}+\frac{1}{4}\norm{(y^0)''}^2_{L^2(0,1)}+\frac{\beta}{4}y^2(0,1)+\frac{\gamma}{4}y_x^2(0,1)\\
		&+\frac{1}{4}\int_{-\tau}^0 |k(s+\tau)|\cdot \norm{B^*y_t(s)}^2_{H}ds.
	\end{aligned}
\end{equation*}
In particular, $E_y(0)>0$.

\underline{Claim $2.$} Let us denote
\begin{equation*}
	r:=\sup\{s\in [0,\delta): \eqref{stima_dal_basso} \text{ holds } \forall\; t\in [0,s)\}
\end{equation*}
and we suppose, by contradiction, that $r<\delta$. Then, by continuity,
\begin{equation*}
	\begin{aligned}
		E_y(r)&=\frac{1}{4}\norm{y_t(r)}^2_{L^2_\frac{1}{a}(0,1)}+\frac{1}{4}\norm{y_{xx}(r)}^2_{L^2(0,1)}+\frac{\beta}{4}y^2(r,1)+\frac{\gamma}{4}y_x^2(r,1)\\
		&+\frac{1}{4}\int_{r-\tau}^r |k(s+\tau)|\cdot \norm{B^*y_t(s)}^2_{H}ds;
	\end{aligned}
\end{equation*}
in particular,
\begin{equation*}
	\frac{1}{4}\norm{y_{xx}(r)}^2_{L^2(0,1)}\le E_y(r) \quad \text{ and } \quad \frac{1}{4}\norm{y_t(r)}^2_{L^2_\frac{1}{a}(0,1)}\le E_y(r).
\end{equation*}
Hence, by Theorem \ref{th E} and using the monotonicity of $h$, we have
\begin{equation*}
	h(\norm{y_{xx}(r)}_{L^2(0,1)})\le h\Bigl (2\sqrt{E_y(r)}\Bigr )\le h\Bigl (2\sqrt{C(T)E_y(0)}\Bigr )<\frac{1}{2}.
\end{equation*}
Therefore, using the definition of $E_y$, by the previous inequality and \eqref{stima nonlin}, we conclude that
\begin{equation*}
	\begin{aligned}
		E_y(r)&=\frac{1}{2}\norm{y_t(r)}^2_{L^2_\frac{1}{a}(0,1)}+\frac{1}{2}\norm{y_{xx}(r)}^2_{L^2(0,1)}+\frac{\beta}{2}y^2(r,1)+\frac{\gamma}{2}y_x^2(r,1)\\
		&-\int_0^1\frac{F(y(r,x))}{a(x)} dx+\frac{1}{2}\int_{r-\tau}^r |k(s+\tau)|\cdot \norm{B^*y_t(s)}^2_{H}ds\\
		&>\frac{1}{4}\norm{y_t(r)}^2_{L^2_\frac{1}{a}(0,1)}+\frac{1}{4}\norm{y_{xx}(r)}^2_{L^2(0,1)}+\frac{\beta}{4}y^2(r,1)+\frac{\gamma}{4}y_x^2(r,1)\\
		&+\frac{1}{4}\int_{r-\tau}^r |k(s+\tau)|\cdot \norm{B^*y_t(s)}^2_{H}ds.
	\end{aligned}
\end{equation*}
This is not possible due to the maximality of $r$; consequently $r=\delta$  and the thesis is proved.
\end{proof}

Thanks to the previous results, we are ready to prove well-posedness and  exponential stability of solutions to \eqref{(P_abstract)} corresponding to sufficiently small initial data.  Note that the following theorem proves both well-posedness and exponential stability via an iterative argument.
Indeed, we first show that Hypothesis \ref{ipo Mbe+WP}.2 is satisfied, for small initial data, on  a finite interval $(0,T).$ Then, we apply the exponential decay estimate of Theorem \ref{Th Duhamel} to show that the solutions remain small enough. Therefore,  we can iterate the argument on successive time intervals obtaining, finally, global solutions exponentially decaying.

\begin{Theorem} \label{global} Assume Hypotheses \ref{ipo k}, \ref{ipo f} and
	\eqref{ipo Mbe} and consider \eqref{(P_abstract)} with initial data $Y^0\in\mathcal{H}_0$ and $\psi\in \mathcal{C}([-\tau,0]; \mathcal{H}_0)$. Then  \eqref{(P_abstract)} satisfies Hypothesis \ref{ipo Mbe+WP}.2 and, if the initial data are sufficiently small,  the corresponding solutions exist and decay  exponentially  in $(0, +\infty)$ according to  \eqref{exp decay}.
\end{Theorem}
	\begin{proof}
		Let us consider a time $T>0$ sufficiently large such that
		\begin{equation*}
			C_{T}:=2M^2e^{2\alpha}(1+\Lambda e^{\omega\tau}b^2)(1+\Lambda e^{2\omega\tau} {b^2})e^{-(\omega-\omega')T}\le 1.
		\end{equation*}
	Furthermore, let $\rho >0$ be such that
	\begin{equation*}
		\rho \le \frac{1}{2\sqrt{C(T)}}h^{-1}\Biggl (\frac{1}{2}\Biggr ),
	\end{equation*}
where $C(\cdot)$ is the function introduced in \eqref{function C}, and consider initial data such that
\begin{equation*}
\norm{(y^0)''}^2_{L^2(0,1)}+{	\norm{y^1}^2_{L^2_\frac{1}{a}(0,1)}+\beta}y^2(0,1)+{\gamma}y_x^2(0,1)+\int_{-\tau}^0 |k(s+\tau)|\cdot \norm{g(s)}^2_{H}ds\le\rho^2.
\end{equation*}
Observe that the previous condition is equivalent to require
\begin{equation}\label{condizioneequivalente}
\norm{Y^0}^2_{\mathcal{H}_0}+\int_{-\tau}^0 |k(s+\tau)|\cdot \norm{g(s)}^2_{H}ds\le\rho^2.
\end{equation}
Now, by Proposition \ref{th esist loc} we know that there exists a local solution $y$ of \eqref{(P)} on a time interval $[0,\delta).$  Without loss of generality, we can assume $\delta < T$ (eventually, we can take a larger $T$). From our assumption on the initial data, on the monotonicity of $h$ and the fact that $\sqrt{C(T)}>1$, we have
\begin{equation*}
	h(\norm{(y^0)''}_{L^2(0,1)})\le h(\rho)\le h\Biggl (\frac{1}{2\sqrt{C(T)}}h^{-1}\Biggl (\frac{1}{2}\Biggr )\Biggr )<\frac{1}{2}.
\end{equation*}
Thus, by Proposition \ref{Prop E}.1, we deduce $E_y(0)>0$. Moreover, from \eqref{stima nonlin}, we obtain
\begin{equation}\label{tondino}
	\begin{aligned}
		E_y(0)&\le \frac{1}{2}\norm{y^1}^2_{L^2(0,1)}+\frac{3}{4}\norm{(y^0)''}^2_{L^2(0,1)}+\frac{\beta}{2}y^2(0,1)+\frac{\gamma}{2}y_x^2(0,1)\\
		&+\frac{1}{2}\int_{-\tau}^0 |k(s+\tau)|\cdot \norm{g(s)}^2_{H}ds\le\rho^2,
	\end{aligned}
\end{equation}
which implies
\begin{equation*}
	h\Bigl (2\sqrt{C(T)E_y(0)}\Bigr )<h\Bigl (2\sqrt{C(T)}\rho\Bigr )<h\Biggl (h^{-1}\Biggl (\frac{1}{2}\Biggr )\Biggr )=\frac{1}{2}.
\end{equation*}
Hence, we can apply Proposition \ref{Prop E} obtaining
\begin{equation}\label{15J1}
	\begin{aligned}
		E_y(t)&> \frac{1}{4}\norm{y_t(t)}^2_{L^2_\frac{1}{a}(0,1)}+\frac{1}{4}\norm{y_{xx}(t)}^2_{L^2(0,1)}+\frac{\beta}{4}y^2(t,1)+\frac{\gamma}{4}y_x^2(t,1)\\
		&+\frac{1}{4}\int_{t-\tau}^t |k(s+\tau)|\cdot \norm{B^*y_t(s)}^2_{H}ds>0,
	\end{aligned}
\end{equation} 
for all $t \in [0,\delta)$; in particular, $E_y(t)\ge\frac{1}{4}\norm{y_t(t)}^2_{L^2_{\frac{1}{a}}(0,1)}$. Thus, we can apply 
Theorem \ref{th E} obtaining
\begin{equation}\label{15J2}
E_y(t)\le C(T)E_y(0), \quad \forall \, t\in [0, \delta)
\end{equation}
(recall that $\delta < T$).
As a consequence, from \eqref{15J1}  and \eqref{15J2}, we have
\begin{equation}\label{3.16}
	\begin{aligned}
		\frac{1}{4}\norm{y_{xx}(t)}^2_{L^2(0,1)} &
		\le \frac{1}{4}\norm{y_t(t)}^2_{L^2_\frac{1}{a}(0,1)}+\frac{1}{4}\norm{y_{xx}(t)}^2_{L^2(0,1)}\\
		&\le \frac{1}{4}\norm{y_t(t)}^2_{L^2_\frac{1}{a}(0,1)}+\frac{1}{4}\norm{y_{xx}(t)}^2_{L^2(0,1)}+\frac{\beta}{4}y^2(t,1)+\frac{\gamma}{4}y_x^2(t,1)\\
		&+\frac{1}{4}\int_{t-\tau}^t |k(s+\tau)|\cdot \norm{B^*y_t(s,x)}^2_{H}ds<E_y(t)\le C(T)E_y(0),
	\end{aligned}
\end{equation}
for every $t\in [0,\delta)$. Then, we can extend the solution in $t=\delta$ and on the whole interval $[0,T]$. In particular, for $t=T$, we get
\begin{equation*}
	\begin{aligned}
		h(\norm{y_{xx}(T)}_{L^2(0,1)})&
		%\le h\Bigl (2\sqrt{E_y(T')}\Bigr )\\
		\le h\Bigl (2\sqrt{C(T)E_y(0)}\Bigr )<\frac{1}{2}.
	\end{aligned}
\end{equation*}
By \eqref{tondino} and \eqref{3.16} we deduce
\begin{equation*}
	\frac{1}{4}\norm{Y(t)}^2_{\mathcal{H}_0}\le E_y(t)\le C(T)E_y(0)\le C(T)\rho^2,\,\,\,\,\,\,\,\,\,\forall\; t\in [0,T],
\end{equation*}
i.e.
\begin{equation*}
	\norm{Y(t)}_{\mathcal{H}_0}\le  C_\rho,\,\,\,\,\,\,\,\,\,\forall\; t\in [0,T],
\end{equation*}
where $C_\rho:=2\sqrt{C(T)}\rho$. Now, without loss of generality, we can assume that $\rho$ is such that $L(C_\rho)<\frac{\omega-\omega'}{2M}$ (eventually choosing a smaller value of $\rho$). Recall that $L(C_\rho)$, $M$, $\omega$ and $\omega'$ are the constants considered in Hypothesis \ref{ipo f}, \eqref{stima S(t)} and \eqref{ipo Mbe}, respectively. Consequently, Hypothesis \ref{ipo Mbe+WP}.2 is satisfied in the interval $[0,T]$. Hence, Theorem \ref{Th Duhamel} gives us the following estimate:
\begin{equation}\label{exp decay1}
	\norm{Y(t)}_{\mathcal{H}_0}\le Me^\alpha \Biggl (	\norm{Y^0}_{\mathcal{H}_0}+\int_0^\tau e^{\omega s}|k(s)|\cdot \norm{\psi(s-\tau)}_{\mathcal{H}_0}ds \Biggr )e^{-\frac{\omega -\omega'}{2}t},
\end{equation}
for any $t\in [0,T]$. By Hypothesis \ref{ipo k}, \eqref{condizioneequivalente} and  the H\"older inequality, it follows
\begin{equation*}
	\begin{aligned}
	\int_0^\tau e^{\omega s}|k(s)|\cdot \norm{\psi(s-\tau)}_{\mathcal{H}_0}ds&\le e^{\omega \tau}\Biggl (	\int_0^\tau \!\!\!\!|k(s)|ds\Biggr )^{\frac{1}{2}}\!\!\!\Biggl (	\int_0^\tau\!\!\!\! |k(s)|\cdot \norm{\psi(s-\tau)}^2_{\mathcal{H}_0}ds\Biggr )^{\frac{1}{2}}\\
	&\le e^{\omega \tau}\sqrt{\Lambda} \rho b.
	\end{aligned}
\end{equation*}
Hence, coming back to \eqref{exp decay1}, we obtain
\begin{equation}\label{3.17'}
	\norm{Y(t)}^2_{\mathcal{H}_0}\le 2M^2e^{2\alpha}\rho^2(1+b^2e^{2\omega\tau}\Lambda)e^{-(\omega -\omega')t},
\end{equation}
for any $t\in [0,T]$. Moreover,
\begin{equation*}
	\int_{T-\tau}^{T}|k(s+\tau)|\cdot \norm{B^*y_t(s)}_{H}^2ds\le 2b^2M^2e^{2\alpha}\Lambda\rho^2 e^{\omega\tau}(1+\Lambda e^{2\omega\tau}b^2)e^{-(\omega -\omega')T}.
\end{equation*}
Since $T$ is chosen such that  $C_{T}\le 1$, by \eqref{3.17'} and the previous inequality, one has
\begin{equation*}
	\norm{Y(T)}^2_{\mathcal{H}_0}+\int_{T-\tau}^{T}|k(s+\tau)|\cdot \norm{B^*y_t(s)}_{H}^2ds\le C_{T}\rho^2\le\rho^2
\end{equation*}
or, equivalently,
\begin{equation*}
\begin{aligned}
	&\norm{y_t(T)}^2_{L^2_\frac{1}{a}(0,1)}+\norm{y_{xx}(T)}^2_{L^2(0,1)}+{\beta}y^2(T,1)\\&+{\gamma}y_x^2(T,1)+\int_{T-\tau}^{T}|k(s+\tau)|\cdot \norm{B^*y_t(s)}_{H}^2ds\le\rho^2.
\end{aligned}
\end{equation*}

We can apply a similar argument on the interval $[T,2T]$, obtaining a solution on the whole interval $[0,2T]$. Iterating the procedure, we obtain a unique global solution of \eqref{(P_abstract)}. Thus, the well posedness Hypothesis \ref{ipo Mbe+WP}.2 is satisfied   and the thesis follows.
	\end{proof}
\section{Delayed  beam equations with source term or integral nonlinearity}\label{Sez 4}
In this section we will apply  the abstract results of Section \ref{section} to two specific problems. To this aim, 
take as   $H$ the Hilbert space $L^2(\mathcal{P})$,  where $\mathcal{P}$ is an open subset strictly contained in $ (0,1)$ and define the bounded linear operator $B$ as
\begin{equation*}
	B:L^2(\mathcal{P})\to L^2_{\frac{1}{a}}(0,1)\,\,\,\,\,\,\,y\mapsto \tilde{y}\,\chi_\mathcal{P},
\end{equation*}
being $\tilde{y}\in L^2(0,1)$ the trivial extension of $y$ outside $\mathcal{P}$. It is easy to verify that
\begin{equation*}
	B^*(\varphi)=\varphi_{|\mathcal{P}}\,\,\,\,\,\,\,\forall\;\varphi\in (L^2_{\frac{1}{a}}(0,1))^*.
\end{equation*}

Hence, $BB^*(\varphi)=\chi_{\mathcal{P}}\varphi$ for all $\varphi\in (L^2_{\frac{1}{a}}(0,1))^*$. 

Moreover, for $f$ we  consider
two types of nonlinearities:
\begin{equation}\label{f1}
f(y(t,x))=	|y(t,x)|^qy(t,x), \end{equation}
with  $q>0$,
or
\begin{equation}\label{f2}
	f(y(t,x))=		\Bigl (\int_0^1|y(t,x)|^2dx\Bigr )^{\frac{p}{2}}y(t,x),
	\end{equation}
with $p\ge 1$. Hence, as a concrete example, we consider
\begin{equation*}
	\begin{cases}
		y_{tt}(t,x)+a(x)y_{xxxx}(t,x)+k(t)\chi_\mathcal{P}(x)y_t(t-\tau,x)=f(y(t,x)), &(t,x)\in Q,\\
		y(t,0)=0,\,\,y_x(t,0)=0, &t>0,\\
		\beta y(t,1)-y_{xxx}(t,1)+y_t(t,1)=0, &t >0,\\
		\gamma y_x(t,1)+y_{xx}(t,1)+y_{tx}(t,1)=0, &t >0,\\
		y(0,x)=y^0(x),\,\,y_t(0,x)=y^1(x),&x\in(0,1),\\
		y_t(s,x)=g(s), &s\in [-\tau,0],
	\end{cases}
\end{equation*}
where $\tau, \beta, \gamma, a, k$ are as in Section \ref{section}, $\chi_\mathcal{P}$ is the characteristic function of the set $\mathcal{P}$ and $f$ is defined as in \eqref{f1} or \eqref{f2}.
In the following we will prove that $f$ satisfies Hypothesis \ref{ipo f}.

First of all, assume \underline{$f(y):=|y|^qy, q>0$}.

Clearly, $f(0)=0$. Now, we will prove the other points of Hypothesis \ref{ipo f}. Observe that
\begin{equation}\label{stima0}
	\bigl | |\alpha|^q\alpha-|\beta|^q\beta \bigr |\le (q+1)\bigl (|\alpha|+|\beta|\bigr )^q|\alpha-\beta|,\,\,\,\,\,\,\,\forall\; \alpha,\beta\in\mathbb{R};
\end{equation}
moreover, for all $v \in H^2_{\frac 1 a, 0}(0,1)$, one has
 \begin{equation}\label{stima}
 |v(x)| \le \int_0^x\int_0^t |v''(s)|dsdt \le \int_0^x \sqrt{t}\|v''\|_{L^2(0,1)}dt \le \frac{2}{3} x^{\frac{3}{2}}\|v''\|_{L^2(0,1)},
 \end{equation}
 for all $x \in (0,1)$.
 Hence,  
\begin{equation*}
	\begin{aligned}
		\int_0^1\frac{1}{a}\bigl ||y|^qy-|z|^qz \bigr |^2\,dx&\le (q+1)^2\int_0^1\frac{1}{a}\bigl (|y|+|z|\bigr )^{2q}|y-z|^2dx\\
		&\le (q+1)^2 C_q\int_0^1\frac{1}{a}\bigl (|y|^{2q}+|z|^{2q}\bigr )|y-z|^2dx,
\end{aligned}
\end{equation*}
by \eqref{stima0}, where
\begin{equation}\label{Cq}
C_q:=
\begin{cases} 2^{2q-1}, &q\ge 1/2,\\
1, & q\in (0,1/2).
\end{cases}
\end{equation}
Then, thanks to the previous inequality, \eqref{stimanormeequi} and \eqref{stima}, one has
\begin{equation*}
\begin{aligned}
		&\int_0^1\frac{1}{a}\bigl ||y|^qy-|z|^qz \bigr |^2\,dx
\le \frac{2}{3}(q+1)^2 C_q \int_0^1 \!\!\!(\|y''\|_{L^2(0,1)}^{2q} + \|z''\|_{L^2(0,1)}^{2q})\frac{|y-z|^2}{a}dx\\
				&\le \frac{2}{3}(q+1)^2 C_q (4C_{HP}+1)\left( \left(\int_0^1|y''|^2dx\right)^q\!+\!\left(\int_0^1\!\!|z''|^2dx\right)^q \right)\int_0^1\!\!|y''-z''|^2dx,
	\end{aligned}
\end{equation*}
for all $y, z \in H^2_{\frac 1 a, 0}(0,1)$.
Now, fixing $r>0$ and taking  $y, z \in H^2_{\frac 1 a, 0}(0,1)$ such that $\|y''\|_{L^2(0,1)}, \|z''\|_{L^2(0,1)} \le r$, we obtain
\begin{equation*}
				\norm{f(y)-f(z)}_{L^2_{\frac{1}{a}}(0,1)}\le L(r)\norm{y''-z''}_{L^2(0,1)},
			\end{equation*}
being $L(r):=\sqrt{\frac{2}{3}(q+1)^2 C_q (4C_{HP}+1) r^{2q}}$.
Moreover, Hypothesis \ref{ipo f}.3 is also satisfied with $h(x):=\left(\frac{2}{3}\right)^q (4C_{HP}+1)x^q$. Indeed, by \eqref{stimanormeequi} and \eqref{stima}, one has
\[
\begin{aligned}
	\langle f(y), y\rangle_{L^2_{\frac{1}{a}}(0,1)}&=\int_0^1\frac{1}{a}|y|^{q+2}dx\le
	\left(\frac{2}{3}\right)^q\|y''\|_{L^2(0,1)}^q\int_0^1\frac{1}{a}|y|^{2}dx\\
	&\le \left(\frac{2}{3}\right)^q (4C_{HP}+1)
	\|y''\|_{L^2(0,1)}^{q+2},
	\end{aligned}
\]
for all $y \in H^2_{\frac 1 a, 0}(0,1)$.
%Finally, introducing for every $t\ge 0$ the following energy functional
%\begin{equation*}
%	\begin{aligned}
%		E_y(t)&:=\frac{1}{2}\int_0^1 \Biggl (\frac{y^2_t(t,x)}{a(x)}+y^2_{xx}(t,x) \Biggr )dx+\frac{\beta}{2}y^2(t,1)+\frac{\gamma}{2}y_x^2(t,1)\\
%		&-\int_0^1\frac{F(y(t))}{a(x)} dx+\frac{1}{2}\int_{t-\tau}^t\int_\mathcal{P} |k(s+\tau)|\cdot |y_t(s,x)|^2ds\,dx,\quad\,\,\,\,\,\,\forall\;t\ge 0,
%	\end{aligned}
%\end{equation*}
%we underline that, under the assumption \eqref{ipo Mbe}, Theorem \ref{Th Duhamel} can be applied to system \eqref{(P_application)}. As a consequence, we obtain well posedness  and an exponential stability result for small initial data.
Since $f$ satisfies Hypothesis \ref{ipo f}, one can apply the results of Section \ref{section} as soon as Hypotheses \ref{ipo k} and  \ref{ipo Mbe+WP} are satisfied.

Now, assume \underline{$f(y):=\Bigl (\int_0^1|y|^2dx\Bigr )^{\frac{p}{2}}y, \; p\ge 1$}.

Again $f(0)=0$.
Moreover, using \eqref{stimanormeequi},  
 one has 
 	\[	 \|u\|_{L^2_{\frac{1}{a}}(0, 1)}^2 \le (4C_{HP}+1) \|u''\|^2_{L^2(0,1)}, \quad \forall \; u \in H^2_{\frac{1}{a},0}(0, 1),
	\]
\begin{equation}\label{stima1}
\begin{aligned}
			&	\norm{f(y)-f(z)}_{L^2_{\frac{1}{a}}(0,1)}^2 = \int_0^1 \frac{1}{a}|\|y\|_{L^2(0,1)}^p y - \|z\|^p_{L^2(0,1)} z|^2 dx\\
		&=\int_0^1{\frac{1}{a}}\bigl (\norm{y}_{L^2(0, 1)}^{p}y-\norm{y}_{L^2(0, 1)}^{p}z+\norm{y}_{L^2(0, 1)}^{p}z-\norm{z}_{L^2(0, 1)}^{p}z \bigr )^2dx\\
		&\le 2\int_0^1{\frac{1}{a}}\norm{y}_{L^2(0, 1)}^{2p}(y-z)^2 dx+
2\int_0^1{\frac{1}{a}}(\norm{y}_{L^2(0, 1)}^{p}-\norm{z}_{L^2(0, 1)}^{p})^2|z|^2dx\\
		&\le 2\norm{y}_{L^2_{\frac{1}{a}}(0, 1)}^{2p}{\left(\max_{x \in [0,1]}a(x)\right)^p}\int_0^1\frac{1}{a}(y-z)^2 dx\\
		&+2(\norm{y}_{L^2(0, 1)}^{p}-\norm{z}_{L^2(0, 1)}^{p})^2 \int_0^1\frac{1}{a} |z|^2dx\\
		&\le 2{\left(\max_{x \in [0,1]}a(x)\right)^p}{(4C_{HP}+1)^{p+1}}\norm{y''}_{L^2(0, 1)}^{2p}\|y''-z''\|^2_{L^2(0,1)}\\
		&+2(\norm{y}_{L^2(0, 1)}^{p}-\norm{z}_{L^2(0, 1)}^{p})^2 \int_0^1\frac{1}{a} |z|^2dx,
				\end{aligned}
			\end{equation}
			for all $y, z \in H^2_{\frac 1 a, 0}(0,1)$.
			Now, consider the term $\norm{y}_{L^2(0, 1)}^{p}-\norm{z}_{L^2(0, 1)}^{p}$. By \eqref{stima0}, one has
\[
		\begin{aligned}
	&	\norm{y}_{L^2(0, 1)}^{p}-\norm{z}_{L^2(0, 1)}^{p}=	\norm{y}_{L^2(0, 1)}^{p-1}\norm{y}_{L^2(0, 1)}-	\norm{z}_{L^2(0, 1)}^{p-1}\norm{z}_{L^2(0, 1)}\\
		&\le p\bigl (\norm{y}_{L^2(0, 1)}+\norm{z}_{L^2(0, 1)} \bigr )^{p-1}|\norm{y}_{L^2(0, 1)}-\norm{z}_{L^2(0, 1)}|\\
		&\le p{C_{\frac{p}{2}}}(\norm{y}_{L^2(0, 1)}^{p-1}+\norm{z}_{L^2(0, 1)}^{p-1} \bigr )|\norm{y}_{L^2(0, 1)}-\norm{z}_{L^2(0, 1)}|\\
		&\le p{C_{\frac{p}{2}}} \left( \max_{x \in [0,1]} a(x)\right)^{\frac{p-1}{2}} \left(\norm{y}_{L^2_{\frac{1}{a}}(0, 1)}^{p-1}+\norm{z}_{L^2_{\frac{1}{a}}(0, 1)}^{p-1} \right)\norm{y-z}_{L^2(0, 1)},
		\end{aligned}	
\]
where {$C_{\frac{p}{2}}$ is defined as in \eqref{Cq}.}
Hence
\[
(\norm{y}_{L^2(0, 1)}^{p}-\norm{z}_{L^2(0, 1)}^{p})^2 \le D_p (\norm{y}_{L^2_\frac{1}{a}(0, 1)}^{2p-2}+\norm{z}_{L^2_\frac{1}{a}(0, 1)}^{2p-2} \bigr )\norm{y-z}_{L^2(0, 1)}^2,
\]
where {$D_p:= 2 p^2 {C_{\frac{p}{2}}} ^2\left( \max_{x \in [0,1]} a(x)\right)^{p-1} $},
and
\[
\begin{aligned}
&2(\norm{y}_{L^2(0, 1)}^{p}-\norm{z}_{L^2(0, 1)}^{p})^2 \int_0^1\frac{1}{a} |z|^2dx\\
& \le 2D_p\max_{x \in [0,1]}a(x)(\norm{y}_{L^2_\frac{1}{a}(0, 1)}^{2p-2}+\norm{z}_{L^2\frac{1}{a}(0, 1)}^{2p-2} \bigr )\norm{y-z}_{L^2_\frac{1}{a}(0, 1)}^2\norm{z}_{L^2_\frac{1}{a}(0, 1)}^2\\
&\le 2D_p\max_{x \in [0,1]}a(x)(4C_{HP}+1)^{p+1}(\norm{y''}_{L^2(0, 1)}^{2p-2}\!+\norm{z''}_{L^2(0, 1)}^{2p-2} \bigr )\|y''\!-\!z''\|^2_{L^2(0,1)}\|z''\|^2_{L^2(0,1)},
\end{aligned}
\]
by \eqref{stimanormeequi}.
By the previous inequality, \eqref{stima1} becomes
\[
\begin{aligned}
&	\norm{f(y)-f(z)}_{L^2_{\frac{1}{a}}(0,1)}^2 
	\le 2{\left(\max_{x \in [0,1]}a(x)\right)^p}{(4C_{HP}\!+\!1)^{p+1}}\norm{y''}_{L^2(0, 1)}^{2p}\|y''\!-\!z''\|^2_{L^2(0,1)}\\
		&+2D_p\max_{x \in [0,1]}a(x)(4C_{HP}\!+\!1)^{p+1}(\norm{y''}_{L^2(0, 1)}^{2p-2}+\norm{z''}_{L^2(0, 1)}^{2p-2} \bigr )\|y''\!-\!z''\|^2_{L^2(0,1)}\|z''\|^2_{L^2(0,1)},\end{aligned}
\]
for all $y, z \in H^2_{\frac 1 a, 0}(0,1)$. Now, fixing $r>0$ and taking  $y, z \in H^2_{\frac 1 a, 0}(0,1)$ such that $\|y''\|_{L^2(0,1)}, \|z''\|_{L^2(0,1)} \le r$, we obtain
\begin{equation*}
				\norm{f(y)-f(z)}_{L^2_{\frac{1}{a}}(0,1)}\le L(r)\norm{y''-z''}_{L^2(0,1)},
			\end{equation*}
being {$L(r):=\sqrt{2\max_{x \in [0,1]}a(x)(4C_{HP}+1)^{p+1}\left(\left(\max_{x \in [0,1]}a(x)\right)^{p-1}+ 2D_p\right)r^{2p} }$.}
Moreover, Hypothesis \ref{ipo f}.3 is also satisfied with \[h(x):=\left(\max_{x\in [0,1]}a(x)\right)^{\frac{p}{2}}(4C_{HP}+1)^{\frac{p}{2}+1}x^p.\] Indeed, using again \eqref{stimanormeequi}, one has
\[
\begin{aligned}
	\langle f(y), y\rangle_{L^2_{\frac{1}{a}}(0,1)}&=\|y\|^p_{L^2(0,1)}\int_0^1\frac{1}{a}|y|^{2}dx\\
	&\le
	\left(\max_{x\in [0,1]}a(x)\right)^{\frac{p}{2}}(4C_{HP}+1)^{\frac{p}{2}+1}\|y''\|_{L^2(0,1)}^{p+2},
			\end{aligned}\]
for all $y \in H^2_{\frac 1 a, 0}(0,1)$. As for the previous example, since $f$ satisfies Hypothesis \ref{ipo f}, one can apply the results of Section \ref{section} as soon as Hypotheses \ref{ipo k} and  \ref{ipo Mbe+WP} are satisfied.
\section{Some extensions}\label{Sez 5}
In this section we study the stability for a  non linear problem governed by a fourth order degenerate operator in divergence form or by a second order operator in divergence or in non divergence form. In every case the function $a$ is (WD) or (SD) and, as in Section \ref{sezione2}, the assumption $K<2$ is only a technical hypothesis (see \cite{ACL}, \cite{CF_ConStab} and \cite{FM}).
\subsection{The nonlinear degenerate Euler-Bernoulli equation in divergence form}
In this section we study the well posedness  and the stability for 
\begin{equation}\label{(P_4div)}
	\begin{cases}
		y_{tt}(t,x)+(ay_{xx})_{xx}(t,x)+k(t)BB^*y_t(t-\tau,x)=f(y(t,x)), &(t,x)\in Q,\\
		y(t,0)=0,&t> 0,\\
		\begin{cases}
			y_x(t,0)=0, &\text{ if } a \text{ is (WD)},\\
			(ay_{xx})(t,0)=0, &\text{ if } a \text{ is (SD)},\\ 
		\end{cases} &t> 0,\\
		\beta y(t,1)-(ay_{xx})_x(t,1)+y_t(t,1)=0, &t >0,\\
		\gamma y_x(t,1)+(ay_{xx})(t,1)+y_{tx}(t,1)=0, &t > 0,\\
		y(0,x)=y^0(x),\,\,y_t(0,x)=y^1(x),&x\in(0,1),
	\end{cases}
\end{equation}
where, as before, $Q:=(0,+\infty) \times (0,1)$, $\beta,\gamma \ge0$, $\tau >0$ is the time delay, $g$ is defined in $[-\tau,0]$ with value on a real  Hilbert space $H$ and $B: H \rightarrow L^2(0,1)$ is a bounded linear operator with adjoint $B^*$. 

As for the system in non divergence form, we consider first of all the problem without delay
 \begin{equation}\label{(P_4divund)}
 	\begin{cases}
 		y_{tt}(t,x)+(ay_{xx})_{xx}(t,x)=0, &(t,x)\in Q,\\
 		y(t,0)=0,&t>0,\\
 		\begin{cases}
 			y_x(t,0)=0, &\text{ if } a \text{ is (WD)},\\
 			(ay_{xx})(t,0)=0, &\text{ if } a \text{ is (SD)},\\ 
 		\end{cases} &t> 0,\\
 		\beta y(t,1)-(ay_{xx})_x(t,1)+y_t(t,1)=0, &t> 0,\\
 		\gamma y_x(t,1)+(ay_{xx})(t,1)+y_{tx}(t,1)=0, &t >0,\\
 		y(0,x)=y^0(x),\,\,y_t(0,x)=y^1(x),&x\in(0,1)
 	\end{cases}
 \end{equation} 
 and we introduce the Hilbert spaces needed for its study. 
 Thus, consider
 \[
 \begin{aligned}
 	V^2_a(0,1):=\{u\in H^1(0,1):& \; u' \text{ is absolutely continuous in [0,1]},\\
 	& \sqrt{a}u''\in L^2(0,1)\}
 \end{aligned}
 \]
 and 
 \[
 \begin{aligned}
 	K^2_a(0,1):&=\{u \in V^2_a(0,1): u(0)=0\}\\
 	&=\{u\in H^1(0,1): u' \text{ is absolutely continuous in [0,1]}, u(0)=0,\\
 	& \quad \;\; \sqrt{a}u''\in L^2(0,1)\},
 \end{aligned}
 \]
 if  $a$ is (WD);
 \[
 \begin{aligned}
 	V^2_a(0,1):=\{u\in H^1(0,1):& \; u' \text{ is locally absolutely continuous in } (0,1],\\
 	& \sqrt{a}u''\in L^2(0,1)\}
 \end{aligned}\]
 and
  \[
 \begin{aligned}
 	K^2_a(0,1):&=\{u \in V^2_a(0,1): u(0)=0\}\\
 	&=\{u\in H^1(0,1): u' \text{ is locally absolutely continuous in } (0,1], \\
 	& \quad \; \;u(0)=0, \sqrt{a}u''\in L^2(0,1)\},
 \end{aligned}\]
 if $a$ is (SD).

 In both cases we consider on $V^2_a(0,1)$ and $K^2_a(0,1)$ the norm
 \begin{equation*}
 	\|u\|^2_{2,a}:= \|u\|^2_{L^2(0,1)}+ \|u'\|^2_{L^2(0,1)}+ \|\sqrt{a}u''\|^2_{L^2(0,1)}, \quad \forall \; u \in V^2_a(0,1),
 \end{equation*}
 which is equivalent  to the following one
 \begin{equation*}
 	\|u \|_{2}^2:= \|u\|^2_{L^2(0,1)}+ \|\sqrt{a}u''\|^2_{L^2(0,1)}, \quad \forall \; u \in V^2_a(0,1)
 \end{equation*}
 (see \cite[Propositions 2.2 and 2.7]{CF_Wentzell}). Moreover, on $K^2_a(0,1)$ we can consider the equivalent norm 
 \[
 \|u\|_{2, \circ}^2 :=  |u'(1)|^2+\|\sqrt{a}u''\|^2_{L^2(0,1)}.
 \]

 The description of the functional setting is completed considering
\[
K_{a,0}^2(0,1):=\{ u \in K^2_a(0,1): u'(0)=0, \text{ when $a$ is (WD)}\},
\]

\[
\begin{aligned}
	\mathcal W_0 (0,1):=\{ u \in K^2_a(0,1): &\; au''\in H^2(0,1) \; \text{and} \; u'(0)=0 \text{ if $a$ is (WD), or}\\&
	\;(au'')(0)=0, \text{  if $a$ is (SD)}\}
\end{aligned}\]
and the product space 
\begin{equation*}
	\mathcal{K}_0:=K^2_{a,0}(0,1)\times L^2(0,1).
\end{equation*}
On $\mathcal{K}_0$ we consider  inner product and norm defined as:
\[
\langle (u,v),(\tilde{u},\tilde{v})\rangle_{\mathcal{K}_0}:=\int_{0}^{1}au''\tilde{u}''dx+\int_{0}^{1}v\tilde{v}\,dx+\beta u(1)\tilde{u}(1)+\gamma u'(1)\tilde{u}'(1)
\]
and 
\[
\|(u,v)\|^2_{\mathcal{K}_0}:=	\int_{0}^{1}a(u'')^2dx+\int_{0}^{1}v^2dx+\beta u^2(1)+\gamma (u')^2(1),
\]
for every $(u,v), (\tilde{u},\tilde{v})\in\mathcal{K}_0$, respectively.

Now, consider the operators $(A_d, D(A_d))$ given by $A_dy:= (ay_{xx})_{xx}$ for all 
$
y\in   D(A_d):=\mathcal W_0(0,1)$ and
$\mathcal{A}_d:D(\mathcal{A}_d)\subset\mathcal{K}_0\to \mathcal{K}_0$ defined as
\[	\mathcal{A}_d:=\begin{pmatrix}
	0 & Id \\
	-A_d & 0
\end{pmatrix},\]
with domain
\[
\begin{aligned}
	D(\mathcal{A}_d):=\{(u,v) \in D(A_d)\times  K^2_{a,0}(0,1): &\; \beta u(1)-(au'')'(1)+v(1)=0,\\ &\; \gamma u'(1)+(au'')(1)+v'(1)=0\}.
\end{aligned}
\]
Thanks to the next  Gauss Green formula
\begin{equation}\label{GF0new}
	\int_{0}^{1}(au'')''v\,dx=[(au'')'v](1) - [au''v'](1)+\int_{0}^{1}au''v''dx,
\end{equation}
for all $(u,v)\in \mathcal W_0 (0,1) \times K^2_{a,0}(0,1)$ (see \cite{CF_ConStab}),
	one can prove that  $(\mathcal{A}_d, D(\mathcal{A}_d))$ is non positive with dense domain and generates a contraction semigroup $(R(t))_{t\ge 0}$ assuming that $a$ is  (WD) or (SD).
Therefore, the following existence theorem holds.
\begin{Theorem}\label{Theorem 2.6new}
	Assume $a$ (WD) or (SD).
	If $(y^0,y^1)\in\mathcal{K}_0$, then there exists a unique mild solution
	\begin{equation*}
		y\in \mathcal{C}^1([0,+\infty);L^2(0,1))\cap \mathcal{C}([0,+\infty);K^2_{a,0}(0,1))
	\end{equation*}
	of (\ref{(P_4divund)}) which depends continuously on the initial data $(y^0,y^1)\in \mathcal{K}_0$. Moreover, if $(y^0,y^1)\in D(\mathcal{A}_d)$, then the solution $y$ is classical, in the sense that
	\begin{equation*}
		y\in \mathcal{C}^2([0,+\infty);L^2(0,1))\cap \mathcal{C}^1([0,+\infty);K^2_{a,0}(0,1))\cap \mathcal{C}([0,+\infty); D(A_d))
	\end{equation*}
	and the equation of (\ref{(P_4divund)}) holds for all $t\ge 0$.
\end{Theorem}
Hence, if  $a$ is (WD) or (SD), a unique mild solution $y$  of (\ref{(P_4divund)}) exists  and we can define the energy associated to the problem \eqref{(P_4divund)} as
\begin{equation*}
		\mathcal{E}_y(t):=\frac{1}{2}\int_0^1 \Biggl (y^2_t(t,x)+a(x)y^2_{xx}(t,x) \Biggr )dx+\frac{\beta}{2}y^2(t,1)+\frac{\gamma}{2}y_x^2(t,1),\quad\,\,\,\,\,\,\forall\;t\ge 0.
\end{equation*}
In addition, if $y$ is classical, then the energy is non increasing and		\[
\frac{d\mathcal{E}_y(t)}{dt}=-y^2_t(t,1)-y^2_{tx}(t,1),\quad\,\,\,\,\,\,\,\forall\;t\ge 0.
\]
In particular, the following stability result holds.

\begin{Theorem}\label{teoremaprincipale2}\cite[Theorem 4.5]{CF_ConStab}
	Assume $a$ (WD) or (SD), $\beta,\gamma>0$ and let $y$ be a mild solution of \eqref{(P_4divund)}. Then there exists a suitable constant $T_0>0$ such that 		\begin{equation}\label{Stabilità2}
		\mathcal{E}_y(t)\le \mathcal{E}_y(0)e^{1-\frac{t}{T_0}},
	\end{equation}
	for all $t\ge T_0$. If $a$ is (WD), then \eqref{Stabilità2} holds also assuming 
 $\beta,\gamma\ge 0$.
		\end{Theorem}
Under the conditions provided in the previous theorem, the exponential decay of solutions for \eqref{(P_4divund)} is uniform. In particular, as a consequence of Theorem \ref{teoremaprincipale2}, we know that the  $\mathcal{C}_0$-semigroup $(R(t))_{t\ge 0}$ generated by $(\mathcal{A}_d, D(\mathcal A_d))$ is exponentially stable in the sense of  \eqref{stima S(t)} for $\beta, \gamma >0$ in the strongly degenerate case and for $\beta, \gamma \ge 0$ in the weakly degenerate one.

Now, as in Section \ref{section}, we consider the delayed problem \eqref{(P_4div)}  and we rewrite it in an abstract form. To this aim,
define $v(t,x)$, $Y^0(x)$, $Y(t,x)$, $\psi(s),\mathcal{B}Y(t),\mathcal{F}(Y(t))$ as in  Section  \ref{section} and, thanks to $(\mathcal A_d, D(\mathcal A_d))$,  \eqref{(P_4div)} can be rewritten as\begin{equation}\label{(P_abstract div)}
	\begin{cases}
		\dot{Y}(t)=\mathcal{A}_dY(t)-k(t)\mathcal{B}Y(t-\tau)+\mathcal{F}(Y(t)), &(t,x) \in Q,\\
		Y(0)=Y^0, &x \in (0,1),\\
		\mathcal{B}Y(s)=\psi(s), &s \in [-\tau,0].
	\end{cases}
\end{equation}
Also in this case, if $Y^0\in\mathcal{K}_0$, the Duhamel formula \eqref{duhamel} still holds substituting the semigroup $(S(t))_{t \ge0}$ with $(R(t))_{t \ge0}$, and, setting
\begin{equation}\label{b'}
b:=	\norm{B}_{\mathcal{L}(H,L^2(0,1))}=\norm{B^*}_{\mathcal{L}(L^2(0,1),H)},
\end{equation}
one has again
\begin{equation*}
	\norm{\mathcal{B}}_{{\mathcal{L}(\mathcal{K}_0)}}=b^2.
\end{equation*}

In order to deal with well posedness and stability for \eqref{(P_4div)}, we make on $k$  the same assumption as before, i.e. Hypothesis \ref{ipo k}; on the other hand, Hypotheses \ref{ipo f} and \ref{ipo Mbe+WP} become:
\begin{Assumptions}\label{ipo f div}
	Let $f:K^2_{a, 0}(0,1)\to L^2(0,1)$ be a continuous function such that
	\begin{enumerate}
	\item $f(0)=0$;
		\item for all $r>0$ there exists a constant $L(r)>0$ such that, for all $u,v\in K^2_{a, 0}(0,1)$ satisfying $\norm{\sqrt{a}u''}_{L^2(0,1)}\le r$ and $\norm{\sqrt{a}v''}_{L^2(0,1)}\le r$, one has
		\begin{equation*}
			\norm{f(u)-f(v)}_{L^2(0,1)}\le L(r)\norm{\sqrt{a}u''-\sqrt{a}v''}_{L^2(0,1)};
		\end{equation*}
		\item there exists a strictly increasing continuous function $h: \R_+ \rightarrow \R _+$ such that 
		\begin{equation}\label{condizione div}
			\langle f(u), u\rangle_{L^2(0,1)}\le h(\norm{\sqrt{a}u''}_{L^2(0,1)})\norm{\sqrt{a}u''}^2_{L^2(0,1)}
		\end{equation}
		for all $u\in K^2_{a, 0}(0,1)$.
	\end{enumerate}
\end{Assumptions}

\begin{Assumptions}\label{ipo Mbe+WP div}
	Suppose that:
	\begin{enumerate}
	\item $\beta, \gamma >0$ if $a$ is (WD) or (SD) and  $\beta, \gamma \ge 0$ if $a$ is (WD);
		\item for any $t>0$ \begin{equation}\label{ipo Mbe div}
			Mb^2e^{\omega \tau}\int_0^t|k(s+\tau)|ds\le \alpha +\omega't
		\end{equation}
		for suitable constants $\alpha\ge 0$ and $\omega'\in [0,\omega)$, where $M$, $\omega$ and $b$ are the constants in \eqref{stima S(t)}, referred to the semigroup $(R(t))_{t\ge 0}$, and \eqref{b'}, respectively;
		\item there exist $T$, $\rho >0$, $C_\rho>0$, with $L(C_\rho)<\frac{\omega - \omega'}{M}$ such that if $Y^0\in\mathcal{K}_0$ and $\psi\in\mathcal{C}([-\tau,0];\mathcal{K}_0)$ satisfy
		\begin{equation}\label{ipo WP div}
			\norm{Y^0}^2_{\mathcal{K}_0}+\int_0^\tau |k(s)|\cdot \norm{g(s-\tau)}^2_{\mathcal{K}_0}ds<\rho^2,
		\end{equation}
		then \eqref{(P_abstract div)} has a unique solution $Y\in\mathcal{C}([0, T);\mathcal{K}_0)$ satisfying $\norm{Y(t)}_{\mathcal{K}_0}\le C_\rho$ for all $t\in [0, T)$.
	\end{enumerate}
\end{Assumptions}

Also in this case, thanks to Hypothesis \ref{ipo f div}, $\mathcal{F}(0)=0$ and for any $r>0$ there exists a constant $L(r)>0$ such that
\begin{equation*}
	\norm{\mathcal{F}(Y)-\mathcal{F}(Z)}_{\mathcal{K}_0}\le L(r)\norm{Y-Z}_{\mathcal{K}_0}
\end{equation*} 
whenever $\norm{Y}_{\mathcal{K}_0}\le r$, $\norm{Z}_{\mathcal{K}_0}\le r$. In particular, 

\[
	\norm{\mathcal{F}(Y)}_{\mathcal{K}_0}\le L(r)\norm{Y}_{\mathcal{K}_0}.
\]

Thanks to the Duhamel formula for \eqref{(P_abstract div)}, we obtain the following theorem whose proof is analogous to the one of Theorem \ref{Th Duhamel}, so we omit it.

\begin{Theorem}\label{Th Duhamel div}
	Assume Hypotheses \ref{ipo f div} and \ref{ipo Mbe+WP div}, $a$ (WD) or (SD)
	and consider  the initial data $(Y^0,\psi)$ satisfying \eqref{ipo WP div}.
	 Then every solution $Y$ of \eqref{(P_abstract div)}  is such that
	 	\begin{equation}\label{exp decay div}
		\norm{Y(t)}_{\mathcal{K}_0}\le Me^\alpha \Biggl (	\norm{Y^0}_{\mathcal{K}_0}+\int_0^\tau e^{\omega s}|k(s)|\cdot \norm{\psi(s-\tau)}_{\mathcal{K}_0}ds \Biggr )e^{-(\omega -\omega'-ML(C_\rho))t},
	\end{equation}
	for any $t\in [0, T)$.
\end{Theorem}

Now, consider the function $F$ defined in \eqref{Fgrande}; Lemma \ref{Lemma 3.1} becomes
\begin{Lemma}
	Assume  Hypothesis \ref{ipo f div} and $a$ (WD) or (SD). Then
	\begin{equation*}
		\Biggl |\int_0^1F(y)dx\Biggr |\le \frac{1}{2}h(\norm{\sqrt{a}y''}_{L^2(0,1)})\norm{\sqrt{a}y''}^2_{L^2(0,1)},
	\end{equation*}
	for all $y\in K^2_{a, 0}(0,1)$.
\end{Lemma}
\begin{proof}
	Fix $y\in K^2_{a, 0}(0,1)$. Proceeding as in Lemma \ref{Lemma 3.1}, one has
	\begin{equation*}
		\begin{aligned}
			\int_0^1F(y)dx&=\int_0^1\int_0^1f(sy)y\,ds\,dx=\int_0^1\langle f(sy), sy\rangle_{L^2(0,1)}\frac{ds}{s}.
		\end{aligned}
	\end{equation*}
	Thus, by \eqref{condizione div},
	\begin{equation*}
		\begin{aligned}
			\Biggl |\int_0^1F(y)dx\Biggr |&\le \int_0^1h(\norm{\sqrt{a}y''}_{L^2(0,1)})s^2\norm{\sqrt{a}y''}^2_{L^2(0,1)}\frac{ds}{s}\\
			&	\le \frac{1}{2}h(\norm{\sqrt{a}y''}_{L^2(0,1)})\norm{\sqrt{a}y''}^2_{L^2(0,1)}.
		\end{aligned}
	\end{equation*}
\end{proof}

Using the function $F$ and under the well posedness assumption \eqref{ipo WP div}, we can define the energy associated to \eqref{(P_4div)} in the following way.
\begin{Definition}
	Let $y$ be a mild solution of (\ref{(P_4div)}) and define its energy  as
	\begin{equation*}
		\begin{aligned}
			E_y(t)&:=\frac{1}{2}\int_0^1 \Biggl (y^2_t(t,x)+a(x)y^2_{xx}(t,x) \Biggr )dx+\frac{\beta}{2}y^2(t,1)+\frac{\gamma}{2}y_x^2(t,1)\\
			&-\int_0^1F(y(t,x)) dx+\frac{1}{2}\int_{t-\tau}^t |k(s+\tau)|\cdot \norm{B^*y_t(s)}^2_{H}ds,\quad\,\,\,\,\,\,\forall\;t\ge 0.
		\end{aligned}
	\end{equation*}
\end{Definition}
For the energy the following estimate holds.

\begin{Theorem}\label{th E div}
	Assume Hypothesis \ref{ipo Mbe+WP div}.2,  $a$ (WD) or (SD)	and let $y$ be a mild solution to \eqref{(P_4div)} defined on a set $[0, T).$ If $E_y(t)\ge\frac{1}{4}\norm{y_t(t)}^2_{L^2(0,1)}$ for any $t\in [0, T),$ then
	\begin{equation*}
		E_y(t)\le C(t)E_y(0)\,\,\,\,\,\,\,\,\,\,\forall\; t\in [0, T),
	\end{equation*}
where $C(\cdot)$ is the function defined in \eqref{function C}.
\end{Theorem}
\begin{proof}
	Let $y$ be a mild solution of \eqref{(P_4div)}. Differentiating formally $E_y$ with respect to $t$, using \eqref{GF0new} and the boundary conditions, we obtain
	\begin{equation*}
		\begin{aligned}
			\frac{dE_y(t)}{dt}&= \int_0^1\Big (
			y_t(t,x)y_{tt}(t,x) +a(x)y_{xx}(t)y_{xxt}(t,x)\Big ) dx\\&+\beta y(t,1)y_t(t,1)+\gamma y_x(t,1)y_{tx}(t,1)\\
			&-\int_0^1f(y(t,x))y_t(t,x) dx+\frac{1}{2}|k(t+\tau)|\cdot \norm{B^*y_t(s)}^2_{H}\\&-\frac{1}{2}|k(t)|\cdot \norm{B^*y_t(t-\tau)}^2_{H}\\
			&= \int_0^1\Big (
			y_t(t,x)y_{tt}(t,x) +(ay_{xx})_{xx}(t,x)y_{t}(t,x)\Big ) dx-(ay_{xx})_x(t,1)y_t(t,1)\\
			&+a(1)y_{xx}(t,1)y_{tx}(t,1)
			+y_t(t,1)[(ay_{xx})_x(t,1)-y_t(t,1)]\\&+y_{tx}(t,1)[-a(1)y_{xx}(t,1)-y_{tx}(t,1)]\\
			&-\int_0^1f(y(t,x))y_t(t,x) dx+\frac{1}{2}|k(t+\tau)|\cdot \norm{B^*y_t(s)}^2_{H}\\&-\frac{1}{2}|k(t)|\cdot \norm{B^*y_t(t-\tau)}^2_{H}\\
			&= \int_0^1\Big (
			y_t(t,x)y_{tt}(t,x) +(ay_{xx})_{xx}(t,x)y_{t}(t,x)\Big ) dx -y_t^2(t,1)-y_{tx}^2(t,1)\\
			&-\int_0^1f(y(t,x))y_t(t,x) dx+\frac{1}{2}|k(t+\tau)|\cdot \norm{B^*y_t(s)}^2_{H}\\&-\frac{1}{2}|k(t)|\cdot \norm{B^*y_t(t-\tau)}^2_{H}.
		\end{aligned}
	\end{equation*}
	Proceeding as in Theorem \ref{th E}, we have
		\begin{equation*}
		\begin{aligned}
			\frac{dE_y(t)}{dt}&=-y_t^2(t,1)-y_{tx}^2(t,1) -k(t)\left\langle B^*y_t(t),B^*y_t(t-\tau)\right\rangle_H\\
			&+\frac{1}{2}|k(t+\tau)|\cdot \norm{B^*y_t(t)}^2_H-\frac{1}{2}|k(t)|\cdot \norm{B^*y_t(t-\tau)}^2_H\\
			& \le \frac{1}{2}(|k(t+\tau)|+|k(t)|) \norm{B^*y_t(t)}^2_H\\
			&\le 2b^2(|k(t+\tau)|+|k(t)|)\frac{1}{4}\norm{y_t(t)}^2_{L^2{(0,1)}}\\
			&\le 2b^2(|k(t+\tau)|+|k(t)|)E_y(t)
	\end{aligned}
	\end{equation*}
	and the thesis follows as in Theorem \ref{th E}.
\end{proof}

We conclude this subsection proving the well posedness assumption (i.e. Hypothesis \ref{ipo Mbe+WP div}.2) for \eqref{(P_abstract div)}. To this aim, observe that Proposition \ref{th esist loc} still holds in this context. On the other hand, the analogue of Proposition \ref{Prop E} is the following one.

\begin{Proposition}\label{Prop E div}
	Assume Hypothesis \ref{ipo k}, $a$ (WD) or (SD) and consider \eqref{(P_abstract div)} with initial data $Y^0\in\mathcal{K}_0$ and $\psi\in \mathcal{C}([-\tau,0]; \mathcal{K}_0).$ Take $T>0$ and let $Y$ be a non trivial solution of \eqref{(P_abstract div)}  defined on $[0,\delta),$ with $\delta\le T$. The following statements hold: 
	\begin{enumerate}
		\item if $h\bigl (\norm{\sqrt{a}(y^0)''}_{L^2(0,1)}\bigr )<\frac{1}{2}$, then $E_y(0)>0$;
		\item if $h\bigl (\norm{\sqrt{a}(y^0)''}_{L^2(0,1)}\bigr )<\frac{1}{2}$ and $h(2\sqrt{C(T)E_y(0)})<\frac{1}{2}$, then
		\begin{equation*}
			\begin{array}{l}
				\displaystyle{E_y(t)>\frac{1}{4}\norm{y_t(t)}^2_{L^2(0,1)}+\frac{1}{4}\norm{\sqrt{a}y_{xx}(t)}^2_{L^2(0,1)}+\frac{\beta}{4}y^2(t,1)+\frac{\gamma}{4}y_x^2(t,1)}
				\\\displaystyle{
					\hspace{5,5 cm}+\frac{1}{4}\int_{t-\tau}^t |k(s+\tau)|\cdot \norm{B^*y_t(s)}^2_{H}ds}
			\end{array}
		\end{equation*}
		for all $t\in [0,\delta)$, being $C(\cdot)$ the function defined in \eqref{function C}. In particular,
		\begin{equation*}
			E_y(t)>\frac{1}{4}\norm{Y(t)}^2_{\mathcal{K}_0},\,\,\,\,\,\,\,\,\,\,\forall\; t\in [0,\delta).
		\end{equation*}
	\end{enumerate}
\end{Proposition}

Thanks to the previous result, we can prove the well posedness hypothesis \eqref{ipo WP div} for  \eqref{(P_4div)}.
\begin{Theorem} Assume Hypotheses \ref{ipo k}, \ref{ipo f div} and
	\eqref{ipo Mbe div}  and consider \eqref{(P_abstract div)} with initial data $Y^0\in\mathcal{K}_0$ and $\psi\in \mathcal{C}([-\tau,0]; \mathcal{K}_0)$. Then  \eqref{(P_abstract div)} satisfies Hypothesis \ref{ipo Mbe+WP div}.2 and, if the initial data are sufficiently small, the  corresponding  solutions exist and decay exponentially in $(0, +\infty)$ according to  \eqref{exp decay div}.
\end{Theorem}

 \subsection{The degenerate second order problems}
The technique developed in the previous sections for nonlinear problems governed by fourth order degenerate operators can also be  applied to the ones governed by second order degenerate operators. This subsection is devoted to present the results in these  last cases.
%In this Section we will apply the technique used before to equations governed by second order degenerate operators. 
\subsubsection{The second order problem in  non divergence form}
As a first step, we will consider the equation in non divergence form
\begin{equation}\label{(P_2nd)}
		\begin{cases}
			y_{tt}(t,x)-a(x)y_{xx}(t,x)+k(t)BB^*y_t(t-\tau,x)=f(y(t,x)), &(t,x)\in Q,\\
			y(t,0)=0, &t >0,\\
			\beta y(t,1)+y_{x}(t,1)+y_t(t,1)=0, &t >0,\\
			y(0,x)=y^0(x),\,\,y_t(0,x)=y^1(x),&x\in(0,1),\\
			B^*y_t(s,x)=g(s), &s\in [-\tau,0].
		\end{cases}
	\end{equation}
With reference to the spaces $L^2_{\frac{1}{a}}(0, 1), H^1_{\frac{1}{a}}(0, 1), H^1_{\frac{1}{a},0}(0, 1)$ defined in Section \ref{sezione2}, following \cite{FM}, we consider 
the operator
\begin{equation*}
		M_{nd}u:=au''
\end{equation*}
with domain
\begin{equation*}
D(M_{nd}):=\left\{u\in \mathcal K^2_{\frac{1}{a}}(0, 1): u(0)=0 \right\},
\end{equation*}
where $\mathcal K^2_{\frac{1}{a}}(0, 1)$ is the Hilbert space
\[\mathcal K^2_{\frac{1}{a}}(0, 1):=\biggl \{u\in H^1_{\frac{1}{a}}(0, 1):au''\in L^2_{\frac{1}{a}}(0, 1) \biggr \}\]
endowed with inner product and  norm given by
\[
\langle u,v\rangle_{\mathcal K^2_{\frac{1}{a}}(0, 1)}:= \langle u,v\rangle_{L^2_{\frac{1}{a}}(0, 1)}+ \langle u',v'\rangle_{L^2(0, 1)}+\langle M_{nd}u,M_{nd}v\rangle_{L^2_{\frac{1}{a}}(0, 1)}\]
and 
\[
\|u\|_{\mathcal K^2_{\frac{1}{a}}(0, 1)}^2:= \|u\|_{L^2_{\frac{1}{a}}(0, 1)}^2+ \|u'\|^2_{L^2(0,1)} + \|\sqrt{a}u''\|^2_{L^2(0,1)},
\]
for all $u,v \in  \mathcal K^2_{\frac{1}{a}}(0, 1)$.
Moreover, we consider the Hilbert space
 \begin{equation*}
	\mathcal{M}_0:=H^1_{\frac{1}{a},0}(0,1)\times L^2_{\frac{1}{a}}(0,1),
\end{equation*}
with natural inner product and norm defined by
\begin{equation*}
	\langle (u,v),(\tilde{u},\tilde{v})\rangle_{\mathcal{M}_0}:=\int_{0}^{1}u'\tilde{u}'dx+\int_{0}^{1}\frac{v\tilde{v}}{a}dx+\beta u(1)\tilde{u}(1)
\end{equation*}
and \begin{equation*}
	\|(u,v)\|^2_{mathcal{M}_0}:=\int_{0}^{1}(u')^2dx+\int_{0}^{1}\frac{v^2}{a}dx+\beta u^2(1),
\end{equation*}
for every $(u,v), \;(\tilde{u},\tilde{v})\in\mathcal{M}_0$. Finally, we consider the matrix operator $\mathcal{M}_{nd}:D(\mathcal{M}_{nd})\subset\mathcal{M}_0\to \mathcal{M}_0$ given by
\begin{equation*}
	\mathcal{M}_{nd}:=\begin{pmatrix}
		0 & Id \\
		M_{nd} & 0
	\end{pmatrix},
\end{equation*}
where
\begin{equation*}
	\begin{aligned}
		D(\mathcal{M}_{nd}):=\{(u,v)\in D(M_{nd})\times H^1_{\frac{1}{a},0}(0,1): &\;\beta u(1)+u'(1)+v(1)=0 \}.
	\end{aligned}
\end{equation*}

Thanks to the Gauss Green formula given  in \cite[Lemma 2.1]{FM}, one can prove that, if $a$ is (WD) or (SD), then  $(\mathcal M_{nd}, D(\mathcal M_{nd}))$ is non positive with dense domain and generates a contraction semigroup  $(\mathcal T(t))_{t \ge 0}$ (see \cite[Theorem 2.1]{FM}). 
Then, considered the undelayed problem
\begin{equation}\label{(P_2undelayed)}
	\begin{cases}
		y_{tt}(t,x)-a(x)y_{xx}(t,x)=0, &(t,x)\in Q,\\
		y(t,0)=0, &t>0,\\
		\beta y(t,1)+y_{x}(t,1)+y_t(t,1)=0, &t >0,\\
		y(0,x)=y^0(x),\,\,y_t(0,x)=y^1(x),&x\in(0,1)
	\end{cases}
\end{equation}
with associated energy
\begin{equation*}
	\mathcal{E}_y(t):=\frac{1}{2}\int_0^1 \Biggl (\frac{y^2_t(t,x)}{a(x)}+y^2_{x}(t,x) \Biggr )dx+\frac{\beta}{2}y^2(t,1),\quad\,\,\,\,\,\,\forall\;t\ge 0,
\end{equation*}
where $\beta \ge 0$, one has the following well posedness and stability result.

\begin{Theorem}\label{Theorem regol3}(see Theorems 2.3 and 3.2 in \cite{FM})
	Assume $a$ (WD) or (SD).
	If $(y^0,y^1)\in\mathcal{M}_0$, then there exists a unique mild solution
	\begin{equation*}
		y\in \mathcal{C}^1([0,+\infty);L^2_{\frac{1}{a}}(0,1))\cap \mathcal{C}([0,+\infty);H^1_{\frac{1}{a},0}(0,1))
	\end{equation*}
	of \eqref{(P_2undelayed)} which depends continuously on the initial data $(y^0,y^1)\in \mathcal{M}_0$. In this case, there exists $T_0>0$ such that 		\begin{equation*}
		\mathcal{E}_y(t)\le \mathcal{E}_y(0)e^{1-\frac{t}{T_0}},
	\end{equation*}
	for all $t\ge T_0$.	
Moreover,	if $(y^0,y^1)\in D(\mathcal{M}_{nd})$, the solution $y$ is classical, in the sense that
	\begin{equation*}
		y\in \mathcal{C}^2([0,+\infty);L^2_{\frac{1}{a}}(0,1))\cap \mathcal{C}^1([0,+\infty);H^1_{\frac{1}{a},0}(0,1))\cap \mathcal{C}([0,+\infty);D(M_{nd}))
	\end{equation*}
	and the equation of \eqref{(P_2undelayed)} holds for all $t\ge 0$. 
	 \end{Theorem}
Under the conditions provided in the previous theorem, the exponential decay of solutions for \eqref{(P_2undelayed)} is uniform; in particular, $(\mathcal T(t))_{t\ge 0}$ satisfies \eqref{stima S(t)}.
\vspace{0.3cm}

In order to treat the nonlinear problem, we assume again Hypothesis \ref{ipo k} on $k$. On the other hand, Hypotheses \ref{ipo f} and   \ref{ipo Mbe+WP}
become:
\begin{Assumptions}\label{ipo f 2ndiv}
	Let $f:H^1_{\frac 1 a, 0}(0,1)\to L^2_{\frac{1}{a}}(0,1)$ be a continuous function such that
	\begin{enumerate}
	\item $f(0)=0$;
		\item for all $r>0$ there exists a constant $L(r)>0$ such that, for all $u,v\in H^1_{\frac 1 a, 0}(0,1)$ satisfying $\norm{u'}_{L^2(0,1)}\le r$ and $\norm{v'}_{L^2(0,1)}\le r$, one has
		\begin{equation*}
			\norm{f(u)-f(v)}_{L^2_{\frac{1}{a}}(0,1)}\le L(r)\norm{u'-v'}_{L^2(0,1)};
		\end{equation*}
		\item there exists a strictly increasing continuous function $h: \R_+ \rightarrow \R_+ $ such that 
		\begin{equation*}
			\langle f(u), u\rangle_{L^2_{\frac{1}{a}}(0,1)}\le h(\norm{u'}_{L^2(0,1)})\norm{u'}^2_{L^2(0,1)},
		\end{equation*}
		for all $u\in H^1_{\frac 1 a, 0}(0,1)$.
	\end{enumerate}
\end{Assumptions}

\begin{Assumptions}\label{ipo Mbe+WP 2ndiv}
	Suppose that:
	\begin{enumerate}
		\item  \eqref{ipo Mbe} is satisfied (in this case the constants are associated to the semigroup $(\mathcal T(t))_{t \ge 0}$);
			\item there exist $T, \rho >0$ and $C_\rho>0$, with $L(C_\rho)<\frac{\omega - \omega'}{M}$ such that if $Y^0\in\mathcal{M}_0$ and $\psi\in\mathcal{C}([-\tau,0];\mathcal{M}_0)$ satisfy
		\begin{equation*}
			\norm{Y^0}^2_{\mathcal{M}_0}+\int_0^\tau |k(s)|\cdot \norm{g(s-\tau)}^2_{\mathcal{M}_0}ds<\rho^2,
		\end{equation*}
		then the abstract problem associated to \eqref{(P_2nd)}  has a unique solution $Y\in\mathcal{C}([0, T);\mathcal{M}_0)$ satisfying $\norm{Y(t)}_{\mathcal{M}_0}\le C_\rho$ for all $t\in [0, T)$.
	\end{enumerate}
\end{Assumptions}

Also in this case, the analogue of Theorem \ref{Th Duhamel} still holds. Moreover, for the function $F$ defined in \eqref{Fgrande} one has:
\begin{Lemma}
	Assume Hypothesis \ref{ipo f 2ndiv} and $a$ (WD) or (SD). Then
	\begin{equation*}
		\Biggl |\int_0^1\frac{F(y)}{a}dx\Biggr |\le \frac{1}{2}h(\norm{y'}_{L^2(0,1)})\norm{y'}^2_{L^2(0,1)},
	\end{equation*}
	for all $y\in H^1_{\frac 1 a, 0}(0,1)$.
\end{Lemma}
We omit the proof since it is similar to the one of Lemma \ref{Lemma 3.1}. As before, we can give the next definition.
\begin{Definition}
	Let $y$ be a mild solution of (\ref{(P_2nd)}) and define its energy  as
	\begin{equation*}
		\begin{aligned}
			E_y(t)&:=\frac{1}{2}\int_0^1 \Biggl (\frac{y^2_t(t,x)}{a(x)}+y^2_{x}(t,x) \Biggr )dx+\frac{\beta}{2}y^2(t,1)\\
			&-\int_0^1\frac{F(y(t,x))}{a(x)} dx+\frac{1}{2}\int_{t-\tau}^t |k(s+\tau)|\cdot \norm{B^*y_t(s)}^2_{H}ds,\quad\,\,\,\,\,\,\forall\;t\ge 0.
		\end{aligned}
	\end{equation*}
\end{Definition}
Moreover, the energy satisfies the following estimate:

\begin{Theorem}\label{th E 2ndiv}
	Assume Hypothesis \ref{ipo Mbe+WP 2ndiv}.2, $a$ (WD) or (SD)	and let $y$ be a mild solution to \eqref{(P_2nd)} defined on $[0, T).$ If $E_y(t)\ge\frac{1}{4}\norm{y_t(t)}^2_{L^2_{\frac 1 a}(0,1)}$ for any $t\in [0, T),$ then
	\begin{equation*}
		E_y(t)\le C(t)E_y(0)\,\,\,\,\,\,\,\,\,\,\forall\; t\in [0, T),
	\end{equation*}
	where $C(\cdot)$ is the function defined in \eqref{function C}.
\end{Theorem}
	\begin{proof}
	Let $y$ be a solution of \eqref{(P_2nd)}. Differentiating formally $E_y$ with respect to $t$, using the Gauss Green formula and the boundary conditions, we have
	\begin{equation*}
		\begin{aligned}
			\frac{dE_y(t)}{dt}&= \int_0^1\Big (
			\frac{y_t(t,x)y_{tt}(t,x)}{a(x)} +y_{x}(t,x)y_{xt}(t,x)\Big ) dx+\beta y(t,1)y_t(t,1)\\
			&-\int_0^1\frac{f(y(t,x))y_t(t,x)}{a(x)} dx+\frac{1}{2}|k(t+\tau)|\cdot \norm{B^*y_t(s)}^2_{H}\\&-\frac{1}{2}|k(t)|\cdot \norm{B^*y_t(t-\tau)}^2_{H}\\
			&= \int_0^1\Big (
			\frac{y_t(t,x)y_{tt}(t,x)}{a(x)} -y_{xx}(t,x)y_{t}(t,x)\Big ) dx+y_{x}(t,1)y_t(t,1)\\
			&+y_t(t,1)[-y_{x}(t,1)-y_t(t,1)]\\
			&-\int_0^1\frac{f(y(t,x))y_t(t,x)}{a(x)} dx+\frac{1}{2}|k(t+\tau)|\cdot \norm{B^*y_t(s)}^2_{H}\\&-\frac{1}{2}|k(t)|\cdot \norm{B^*y_t(t-\tau)}^2_{H}\\
			&= \int_0^1\Big (
			\frac{y_t(t,x)y_{tt}(t,x)}{a(x)} -y_{xx}(t,x)y_{t}(t,x)\Big ) dx -y_t^2(t,1)\\
			&-\int_0^1\frac{f(y(t,x))y_t(t,x)}{a(x)} dx+\frac{1}{2}|k(t+\tau)|\cdot \norm{B^*y_t(s)}^2_{H}\\&-\frac{1}{2}|k(t)|\cdot \norm{B^*y_t(t-\tau)}^2_{H}.
		\end{aligned}
	\end{equation*}
	Again one has
	\begin{equation*}
		\frac{dE_y(t)}{dt}\le 2b^2(|k(t+\tau)|+|k(t)|)E_y(t)
	\end{equation*}
	and the thesis follows using the Gronwall Lemma.
\end{proof}

It remains to prove the well posedness assumption, i.e. Hypothesis \ref{ipo Mbe+WP 2ndiv}.2, for \eqref{(P_2nd)}. Again Proposition \ref{th esist loc} holds in this context and the analogue of Proposition \ref{Prop E} is the following:

\begin{Proposition}\label{Prop E 2ndiv}
	Assume Hypothesis \ref{ipo f 2ndiv} and $a$ (WD) or (SD). Take $T>0$ and let $Y$ be a non trivial solution of the abstract problem with initial data $Y^0\in\mathcal{M}_0$ and $\psi\in \mathcal{C}([-\tau,0],\mathcal{M}_0),$ defined on $[0,\delta),$ with $\delta\le T$. The following statements hold: 
		\begin{enumerate}
		\item if $h\bigl (\norm{(y^0)'}_{L^2(0,1)}\bigr )<\frac{1}{2}$, then $E_y(0)>0$;
		\item if $h\bigl (\norm{(y^0)'}_{L^2(0,1)}\bigr )<\frac{1}{2}$ and $h( 2\sqrt{C(T)E_y(0)})<\frac{1}{2}$, then
		\begin{equation*}
			\begin{array}{l}
				\displaystyle{E_y(t)>\frac{1}{4}\norm{y_t(t)}^2_{L^2_{\frac{1}{a}}(0,1)}+\frac{1}{4}\norm{y_{x}(t)}^2_{L^2(0,1)}+\frac{\beta}{4}y^2(t,1)}
				\\\displaystyle{
					\hspace{5,5 cm}+\frac{1}{4}\int_{t-\tau}^t |k(s+\tau)|\cdot \norm{B^*y_t(s)}^2_{H}ds}
			\end{array}
		\end{equation*}
		for all $t\in [0,\delta)$, being $C(\cdot)$ the function defined in \eqref{function C}. In particular,
		\begin{equation*}
			E_y(t)>\frac{1}{4}\norm{Y(t)}^2_{\mathcal{M}_0},\,\,\,\,\,\,\,\,\,\,\forall\; t\in [0,\delta).
		\end{equation*}
	\end{enumerate}
\end{Proposition}

As a consequence, one has that Hypothesis \ref{ipo Mbe+WP 2ndiv}.2 for system \eqref{(P_2nd)} still holds:
\begin{Theorem}
	Assume Hypotheses \ref{ipo k}, \ref{ipo f 2ndiv} and Hypothesis \ref{ipo Mbe+WP 2ndiv}.1. Consider the abstract problem associated to \eqref{(P_2nd)}, with initial data $Y^0\in\mathcal{M}_0$ and $\psi\in \mathcal{C}([-\tau,0]; \mathcal{M}_0)$. Then the problem satisfies Hypothesis \ref{ipo Mbe+WP 2ndiv}.2 and, if the initial data are sufficiently small, the corresponding solutions exist and decay exponetially according to the following law
	\begin{equation*}		\norm{Y(t)}_{\mathcal{M}_0}\le Me^\alpha \Biggl (	\norm{Y^0}_{\mathcal{M}_0}+\int_0^\tau e^{\omega s}|k(s)|\cdot \norm{\psi(s-\tau)}_{\mathcal{M}_0}ds \Biggr )e^{-(\omega -\omega'-ML(C_\rho))t},
	\end{equation*}
	for any $t\in (0, +\infty)$.
\end{Theorem}

\subsubsection{The second order problem in divergence form}
Now, we consider the problem governed by a second order degenerate operator in divergence form. In particular,  we consider the following problem:
\begin{equation}\label{(P_2div)}
	\begin{cases}
		y_{tt}(t,x)-(ay_{x})_{x}(t,x)+k(t)BB^*y_t(t-\tau,x)=f(y(t,x)), &(t,x)\in Q,\\
		\begin{cases}
			y(t,0)=0, &\text{ if } a \text{ is (WD)},\\
			\lim_{x\to 0}(ay_{x})(t,x)=0, &\text{ if } a \text{ is (SD)},\\ 
		\end{cases} &t>0,\\
		\beta y(t,1)+y_x(t,1)+y_t(t,1)=0, &t >0,\\
		y(0,x)=y^0(x),\,\,y_t(0,x)=y^1(x),&x\in(0,1),
	\end{cases}
\end{equation}
where $\beta >0$.
Following \cite{ACL}, we consider the following space
\[
\begin{aligned}
	Q^1_a(0,1):=\{u\in L^2(0,1):& \; u \text{ is locally absolutely continuous in (0,1]},\\
	& \sqrt{a}u'\in L^2(0,1)\},
\end{aligned}
\]
with inner product 
\begin{equation*}
	\left\langle u,v\right\rangle^2_{1,a}:= \int_0^1uv\,dx+\int_0^1au'v'dx
\end{equation*} and norm
\begin{equation*}
	\|u \|_{1,a}^2:= \|u\|^2_{L^2(0,1)}+ \|\sqrt{a}u'\|^2_{L^2(0,1)},
\end{equation*}
for all $u, v\in V^1_a(0,1)$.
Next, denote by  $W^1_a(0,1)$ the space $Q^1_a(0,1)$ itself if $a$ is (SD) and, if $a$ is (WD), the closed subspace of $Q^1_a(0,1)$ consisting of all the functions $u \in Q^1_a(0,1)$ such that $u(0)=0$.
Moreover, define
\[
\begin{aligned}
	Q^2_a(0,1):=\{u\in Q^1_a(0,1):au'\in H^1(0,1)\},
\end{aligned}
\]
\begin{equation*}
	W^2_{a}(0,1):=W^1_{a}(0,1)\cap Q^2_{a}(0,1)
\end{equation*}
and
\begin{equation*}
	\mathcal{N}_0:=W^1_{a}(0,1)\times L^2(0,1).
\end{equation*}
On $Q^2_a(0,1)$ and  $\mathcal{N}_0$ consider inner products and norms given by
\[
\left\langle u,v\right\rangle^2_{2,a}:= \int_0^1uv\,dx+\int_0^1au'v'dx + \int_0^1(au')'(av')'dx,
\]
\[
	\|u \|_{2,a}^2:= \|u\|^2_{L^2(0,1)}+ \|\sqrt{a}u'\|^2_{L^2(0,1)} + \|(au')'\|^2_{L^2(0,1)},
\]
for every $u,v \in Q^2_a(0,1)$ and
\begin{equation*}
	\langle (u,v),(\tilde{u},\tilde{v})\rangle_{\mathcal{N}_0}:=\int_{0}^{1}au'\tilde{u}'dx+\int_{0}^{1}v\tilde{v}dx+\beta a(1) u(1)\tilde{u}(1),
\end{equation*}
\begin{equation*}
	\|(u,v)\|^2_{\mathcal{N}_0}:=\int_{0}^{1}(\sqrt{a}u')^2dx+\int_{0}^{1}v^2dx+\beta a(1) u^2(1)
\end{equation*}
for every $(u,v), \;(\tilde{u},\tilde{v})\in\mathcal{N}_0$.  Finally, setting 
\begin{equation*}		
D(\mathcal{M}_{d}):=\{(u,v)\in W^2_{a}(0,1)\times W^1_{a}(0,1): \beta u(1)+u'(1)+v(1)=0 \},
\end{equation*}
we can define the operator matrix
$\mathcal{M}_{d}:D(\mathcal{M}_{d})\subset\mathcal{N}_0\to \mathcal{N}_0$ given by
\begin{equation*}
	\mathcal{M}_{d}:=\begin{pmatrix}
		0 & Id \\
		M_{d} & 0
	\end{pmatrix},
\end{equation*}
where $M_d u:= (au')'$ for all $u \in W^2_{a}(0,1)$.
As proved in \cite{ACL}, one has that
$(\mathcal M_{d}, D(\mathcal M_{d}))$  generates a contraction semigroup  $(V(t))_{t \ge 0}$ in $\mathcal N_0$ and the following well posedness  and stability result holds for the undelayed problem
 \begin{equation}\label{(P_2d)}
	\begin{cases}
		y_{tt}(t,x)-(ay_{x})_{x}(t,x)=0, &(t,x)\in Q,\\
		y(t,0)=0,&t>0,\\
		\begin{cases}
			y_x(t,0)=0, &\text{ if } a \text{ is (WD)},\\
			\lim_{x\to 0}(ay_{x})(t,x)=0, &\text{ if } a \text{ is (SD)},\\ 
		\end{cases} &t>0,\\
		\beta y(t,1)+y_x(t,1)+y_t(t,1)=0, &t >0,\\
		y(0,x)=y^0(x),\,\,y_t(0,x)=y^1(x),&x\in(0,1).
	\end{cases}
\end{equation} 

\begin{Theorem}(see \cite[Corollary 4.2 and Theorem 4.5]{ACL})	Assume $a$ (WD) or (SD).
	If $(y^0,y^1)\in \mathcal N_0$, then there exists a unique mild solution
	\begin{equation*}
		y\in \mathcal{C}^1([0,+\infty);\mathcal N_0)\cap \mathcal{C}([0,+\infty);	D(\mathcal{M}_{d}))
	\end{equation*}
	of \eqref{(P_2d)} which depends continuously on the initial data $(y^0,y^1)\in \mathcal{N}_0$. If $(y^0,y^1)\in D(\mathcal{M}_{d})$, then the solution $y$ is classical, in the sense that
	\begin{equation*}
		y\in \mathcal{C}^2([0,+\infty);L^2(0,1))\cap \mathcal{C}^1([0,+\infty);W^1_{a}(0,1))\cap \mathcal{C}([0,+\infty);W^2_{a}(0,1)).
	\end{equation*}
	Moreover, if $\beta >0$ and $y$ is a mild solution of \eqref{(P_2d)}, then there exists a suitable constant $C>0$ such that 	
	the energy associated to \eqref{(P_2d)} given by
\begin{equation*}
	\mathcal{E}_y(t):=\frac{1}{2}\int_0^1 \Biggl (y^2_t(t,x)+a(x)y^2_{x}(t,x) \Biggr )dx+\frac{\beta}{2}a(1)y^2(t,1),\quad\,\,\,\,\,\,\forall\;t\ge 0,
\end{equation*}
decays exponentially, i.e. $\exists \; T_0>0$ such that
	\begin{equation*}
		\mathcal{E}_y(t)\le \mathcal{E}_y(0)e^{1-\frac{t}{T_0}},
	\end{equation*}
	for all $t\ge T_0$.
\end{Theorem}

Thus, if $\beta >0$, the solutions of \eqref{(P_2d)} decay exponentially uniformly and the semigroup $(V(t))_{t\ge 0}$ is exponentially stable, i.e. it satisfies an inequality similar to \eqref{stima S(t)}.

In order to study the delayed problem \eqref{(P_2div)}  we consider the following assumptions:
\begin{Assumptions}\label{ipo f 2div}
	Let $f:W^1_{a}(0,1)\to L^2(0,1)$ be a continuous function such that
	\begin{enumerate}
	\item  $f(0)=0$ ;
		\item for all $r>0$ there exists a constant $L(r)>0$ such that, for all $u,v\in W^1_{a}(0,1)$ satisfying $\norm{\sqrt{a}u'}_{L^2(0,1)}\le r$ and $\norm{\sqrt{a}v'}_{L^2(0,1)}\le r$, one has
		\begin{equation*}
			\norm{f(u)-f(v)}_{L^2(0,1)}\le L(r)\norm{\sqrt{a}u'-\sqrt{a}v'}_{L^2(0,1)}; 
		\end{equation*}
		\item there exists a strictly increasing continuous function $h: \R_+ \rightarrow \R_+ $ such that 
		\begin{equation*}
			\langle f(u), u\rangle_{L^2(0,1)}\le h(\norm{\sqrt{a}u'}_{L^2(0,1)})\norm{\sqrt{a}u'}^2_{L^2(0,1)},
		\end{equation*}
		for all $u\in W^1_{a}(0,1)$.
	\end{enumerate}
\end{Assumptions}

\begin{Assumptions}\label{ipo Mbe+WP 2div}
	Suppose that:
	\begin{enumerate}
		\item \eqref{ipo Mbe} is satisfied (in this case the constants are associated to the semigroup $(V(t))_{t\ge 0}$);
		\item there exist $T, \rho >0$, $C_\rho>0$, with $L(C_\rho)<\frac{\omega - \omega'}{M}$ such that if $Y^0\in\mathcal{N}_0$ and $\psi\in\mathcal{C}([-\tau,0];\mathcal{N}_0)$ satisfy
		\[			\norm{Y^0}^2_{\mathcal{N}_0}+\int_0^\tau |k(s)|\cdot \norm{g(s-\tau)}^2_{\mathcal{N}_0}ds<\rho^2,
		\]
		then the abstract system associated to \eqref{(P_2div)} has a unique solution $Y\in\mathcal{C}([0, T);\mathcal{N}_0)$ satisfying $\norm{Y(t)}_{\mathcal{N}_0}\le C_\rho$ for all $t\in [0, T)$.
	\end{enumerate}
\end{Assumptions}

Theorem \ref{Th Duhamel} still holds and the function $F$ defined in \eqref{Fgrande} satisfies the following estimate.

\begin{Lemma}
	Assume Hypothesis \ref{ipo f 2div} and $a$ (WD) or (SD). Then
	\begin{equation*}
		\Biggl |\int_0^1F(y)dx\Biggr |\le \frac{1}{2}h(\norm{\sqrt{a}y'}_{L^2(0,1)})\norm{\sqrt{a}y'}^2_{L^2(0,1)},
	\end{equation*}
	for all $y\in W^1_a(0,1)$.
\end{Lemma}

Under the well posedness  Hypothesis \ref{ipo Mbe+WP 2div}.2 and using the function $F$, we can define the energy associated to \eqref{(P_2div)} in the following way:
\begin{Definition}
	Let $y$ be a mild solution of (\ref{(P_2div)}) and define its energy  as
	\begin{equation*}
		\begin{aligned}
			E_y(t)&:=\frac{1}{2}\int_0^1 \Biggl (y^2_t(t,x)+a(x)y^2_{x}(t,x) \Biggr )dx+\frac{\beta}{2}a(1)y^2(t,1)\\
			&-\int_0^1F(y(t,x)) dx+\frac{1}{2}\int_{t-\tau}^t |k(s+\tau)|\cdot \norm{B^*y_t(s)}^2_{H}ds,\quad\,\,\,\,\,\,\forall\;t\ge 0.
		\end{aligned}
	\end{equation*}
\end{Definition}
The next estimate holds.
\begin{Theorem}\label{th E 2div}
	Assume Hypothesis \ref{ipo Mbe+WP 2div}.2, $a$ (WD) or (SD)	and let $y$ be a mild solution to \eqref{(P_2div)} defined on $[0, T).$ If $E_y(t)\ge\frac{1}{4}\norm{y_t(t)}^2_{L^2(0,1)}$ for any $t\in [0, T),$ then
	\begin{equation*}
		E_y(t)\le C(t)E_y(0)\,\,\,\,\,\,\,\,\,\,\forall\; t\in [0, T),
	\end{equation*}
	where $C(\cdot)$ is the function defined in \eqref{function C}.
\end{Theorem}
\begin{proof}
	Let $y$ be a solution to \eqref{(P_2div)}. Proceeding as before, one has	\begin{equation*}
		\begin{aligned}
			\frac{dE_y(t)}{dt}&= \int_0^1\Big (
			y_t(t,x)y_{tt}(t,x) +a(x)y_{x}(t,x)y_{xt}(t,x)\Big ) dx+\beta a(1)y(t,1)y_t(t,1)\\
			&-\int_0^1f(y(t,x))y_t(t,x) dx+\frac{1}{2}|k(t+\tau)|\cdot \norm{B^*y_t(s)}^2_{H}\\&-\frac{1}{2}|k(t)|\cdot \norm{B^*y_t(t-\tau)}^2_{H}\\
			&= \int_0^1\Big (
			y_t(t,x)y_{tt}(t,x) -(ay_{x})_{x}(t,x)y_{t}(t,x)\Big ) dx\\
			&+(ay_{x})(t,1)y_t(t,1)+a(1)y_t(t,1)[-y_x(t,1)-y_t(t,1)]\\
			&-\int_0^1f(y(t,x))y_t(t,x) dx+\frac{1}{2}|k(t+\tau)|\cdot \norm{B^*y_t(s)}^2_{H}\\&-\frac{1}{2}|k(t)|\cdot \norm{B^*y_t(t-\tau)}^2_{H}\\
			&= \int_0^1\Big (
			y_t(t,x)y_{tt}(t,x) -(ay_{x})_{x}(t,x)y_{t}(t,x)\Big ) dx -a(1)y_t^2(t,1)\\
			&-\int_0^1f(y(t,x))y_t(t,x) dx+\frac{1}{2}|k(t+\tau)|\cdot \norm{B^*y_t(s)}^2_{H}\\&-\frac{1}{2}|k(t)|\cdot \norm{B^*y_t(t-\tau)}^2_{H}.
		\end{aligned}
	\end{equation*}
	Since, also in this case, one can prove that
	\begin{equation*}
		\frac{dE_y(t)}{dt}\le 2b^2(|k(t+\tau)|+|k(t)|)E_y(t),
	\end{equation*}
the thesis follows as in Theorem \ref{th E}.
\end{proof}

In this case Proposition \ref{Prop E} becomes

\begin{Proposition}\label{Prop E 2div}
	Assume Hypothesis \ref{ipo f 2div} and $a$ (WD) or (SD). Take $T>0$ and let $Y$ be a non trivial solution of the abstract problem associated to \eqref{(P_2div)}, with initial data $Y^0\in\mathcal{N}_0$ and $\psi\in \mathcal{C}([-\tau,0]; \mathcal{N}_0),$ defined on $[0,\delta),$ with $\delta\le T$. The following statements hold: 
	\begin{enumerate}
		\item if $h\bigl (\norm{\sqrt{a}(y^0)'}_{L^2(0,1)}\bigr )<\frac{1}{2}$, then $E_y(0)>0$;
		\item if $h\bigl (\norm{\sqrt{a}(y^0)'}_{L^2(0,1)}\bigr )<\frac{1}{2}$ and  $h(2\sqrt{C(T)E_y(0)})<\frac{1}{2}$, then
		\begin{equation*}
			\begin{array}{l}
				\displaystyle{E_y(t)>\frac{1}{4}\norm{y_t(t)}^2_{L^2(0,1)}+\frac{1}{4}\norm{\sqrt{a}y_{x}(t)}^2_{L^2(0,1)}+\frac{\beta}{4}a(1)y^2(t,1)}
				\\\displaystyle{
					\hspace{5,5 cm}+\frac{1}{4}\int_{t-\tau}^t |k(s+\tau)|\cdot \norm{B^*y_t(s)}^2_{H}ds}
			\end{array}
		\end{equation*}
		for all $t\in [0,\delta)$, being $C(\cdot)$ the function defined in \eqref{function C}. In particular,
		\begin{equation*}
			E_y(t)>\frac{1}{4}\norm{Y(t)}^2_{\mathcal{N}_0},\,\,\,\,\,\,\,\,\,\,\forall\; t\in [0,\delta).
		\end{equation*}
	\end{enumerate}
\end{Proposition}
As a consequence we have the last theorem.
\begin{Theorem}
	Assume Hypotheses \ref{ipo k}, \ref{ipo f 2div} and Hypothesis \ref{ipo Mbe+WP 2div}.1. Consider the abstract problem associated to \eqref{(P_2div)}, with initial data $Y^0\in\mathcal{N}_0$ and $\psi\in \mathcal{C}([-\tau,0];\mathcal{N}_0)$. Then the problem satisfies Hypothesis \ref{ipo Mbe+WP 2div}.2 and, if the initial data are sufficiently small, the corresponding solutions exist and decay exponetially according to the following law
	\begin{equation*}		\norm{Y(t)}_{\mathcal{N}_0}\le Me^\alpha \Biggl (	\norm{Y^0}_{\mathcal{N}_0}+\int_0^\tau e^{\omega s}|k(s)|\cdot \norm{\psi(s-\tau)}_{\mathcal{N}_0}ds \Biggr )e^{-(\omega -\omega'-ML(C_\rho))t},
	\end{equation*}
	for any $t\in (0, +\infty)$.
\end{Theorem}

\section{Acknowledgments}
This work was started while Alessandro Camasta was visiting the University of L'Aquila during his Ph.D. The authors thank the University of Bari Aldo Moro and the University of L'Aquila for this opportunity.

The authors are members of  {\it Gruppo Nazionale per l'Analisi Ma\-te\-matica, la Probabilit\`a e le loro Applicazioni (GNAMPA)} of the Istituto Nazionale di Alta Matematica (INdAM). They are partially supported by INdAM GNAMPA Project {\it ``Modelli differenziali per l'evoluzione del clima e i suoi impatti"} (CUP E53C22001930001). 

A. Camasta is also partially supported by the project {\it Mathematical models for interacting dynamics on networks (MAT-DYN-NET)
CA18232} and by INdAM GNAMPA Project {\it ``Analysis, control and inverse problems for evolution
equations arising in climate science"} (CUP E53C23001670001). He is also member of {\it UMI ``Modellistica Socio-Epidemiologica (MSE)''}

G. Fragnelli  is also partially supported by FFABR {\it Fondo per il finanziamento delle attivit\`a base di ricerca} 2017,   by INdAM GNAMPA Project {\it ``Analysis, control and inverse problems for evolution
equations arising in climate science"} (CUP E53C23001670001), and by the project {\it Mathematical models for interacting dynamics on networks (MAT-DYN-NET)
CA18232}. She is also a member of {\it UMI ``Modellistica Socio-Epidemiologica (MSE)''} and {\it UMI "CliMath"}.

C. Pignotti is partially supported by PRIN 2022  (2022238YY5) {\it Optimal control problems: analysis,
approximation and applications}, PRIN-PNRR 2022 (P20225SP98) {\it Some mathematical approaches to climate change and its impacts}, and by INdAM GNAMPA Project {\it ``Modelli alle derivate parziali per interazioni multiagente non 
simmetriche"}(CUP E53C23001670001). She is also a member of {\it UMI "CliMath"}.

\end{document}